\documentclass{amsart}
\usepackage[mathscr]{eucal}
\usepackage{amssymb, amsmath,array, amscd}
\usepackage{url,enumerate}
\usepackage[dvips]{graphicx}

\usepackage{color}

\def\defi[#1]{{\bf{#1}}}

\input xy
\xyoption{all}

\newtheorem{thm}{Theorem}[section]
\newtheorem{THM}{Theorem}

\newtheorem{prop}[thm]{Proposition}

\newtheorem{lemma}[thm]{Lemma}

\newtheorem{cor}[thm]{Corollary}

\theoremstyle{definition}
\newtheorem{definition}[thm]{Definition}
\newtheorem{remark}[thm]{Remark}

\newtheorem{example}{Example}[section]

\numberwithin{equation}{section}

\newcommand{\Par}{\mathrm{Par}}
\newcommand{\z}{\boldsymbol{z}}
\newcommand{\Z}{\boldsymbol{Z}}
\newcommand{\bc}{\boldsymbol{c}}
\newcommand{\bC}{\boldsymbol{C}}
\newcommand{\ba}{\boldsymbol{a}}
\newcommand{\bb}{\boldsymbol{b}}
\newcommand{\yt}{r}
\newcommand{\Hig}{\mathrm{Higgs}}

\begin{document}
\title{Flat parabolic vector bundles on elliptic curves}

\author[T. Fassarella]{Thiago Fassarella}
\address{\color{black}Universidade Federal Fluminense, Rua M\'ario Santos Braga S/N, Niter\'oi, RJ, Brasil}
\email{\color{black}tfassarella@id.uff.br}

\author[F. Loray]{Frank Loray}
\address{Univ Rennes 1, CNRS, Institut de Recherche en Math\'ematique de Rennes, IRMAR, UMR 6625, Rennes, France}
\email{frank.loray@univ-rennes1.fr}

\thanks{
The first author is supported by CNPq, Proc. 234895/2013-6. The second author is supported by CNRS, ANR-16-CE40-0008  Foliage. The authors also thank  Brazilian-French Network in Mathematics and MATH AmSud for support.}

\begin{abstract}
We describe the moduli space of logarithmic rank $2$ connections on elliptic curves
with $2$ poles. 
\end{abstract}

\maketitle

\tableofcontents

\section{Introduction}

In this paper, we investigate the geometry of certain moduli spaces of connections on curves $C$.
We consider pairs $(E,\nabla)$ where $E\to C$ is a rank $2$ vector bundle and $\nabla:E\to E\otimes\Omega^1_C(D)$
is a logarithmic connection with (reduced) polar divisor $D=t_1+\cdots+t_n$. 
Once we have prescribed the base curve $(C,D)$,
the trace connection $(\det(E),{\rm tr}(\nabla))$ and the eigenvalues $\nu=(\nu_1^{\pm},\ldots,\nu_n^{\pm})$ 
of the residual connection matrix at each pole, then we can define the moduli space ${\rm Con}^{\nu}(C,D)$ 
of those pairs $(E,\nabla)$ up to isomorphism. For a generic choice of $\nu$ (compatible with ${\rm tr}(\nabla)$)
all connections $(E,\nabla)$ are irreducible and the moduli ${\rm Con}^{\nu}(C,D)$ can be constructed 
as a GIT quotient (see \cite{Ni}): it is a smooth irreducible quasi-projective variety of dimension $2N$
where $N=3g-3+n$ is the dimension of deformation of the base curve and  $g$ denotes the genus of the curve. Moreover, the variety 
${\rm Con}^{\nu}(C,D)$ admits a holomorphic symplectic structure (see \cite{Boalch}) which turns to be 
algebraic (see \cite{IIS1,Inaba}): there is a rational $2$-form $\omega$ which is regular and having maximal rank $N$
on ${\rm Con}^{\nu}(C,D)$.

It is natural to consider the forgetful map $\pi:(E,\nabla)\mapsto (E,{\bf p})$ which to
a connection associates the underlying parabolic bundle: the parabolic data ${\bf p}=(p_1,\ldots,p_n)$ consists
of the $\nu_i^+$-eigenline $p_i\subset E\vert_{t_i}$ for each pole. The moduli space ${\rm Bun}(C,D)$
of those parabolic bundles admitting a connection is $N$-dimensional and the map $\pi$ above
turns to be Lagrangian, i.e. its fibers are Lagrangian submanifolds. However, ${\rm Bun}(C,D)$
is not a variety, but a non Hausdorff scheme; it is a finite union of projective varieties patched together 
along Zariski open sets. Over the open subset of simple bundles 
(i.e. without automorphisms), the Lagrangian fibration $\pi$ is an affine $\mathbb A^{N}$-bundle
whose linear part is the cotangent bundle $T^*{\rm Bun}(C,D)$, and the symplectic structure
comes from Liouville form. It is mainly this heuristic picture that we want to describe in a particular case.

The picture is very well known in the case $(g,n)=(0,4)$, 
since ${\rm Con}^{\nu}(C,D)$ corresponds to the Okamoto space of initial conditions for Painlev\'e VI equation
in this case (see \cite{IIS}). The case $(0,n)$ has been studied in \cite{AL,Oblezin,LS,KS} and corresponds to  Garnier systems. The case $(1,1)$ has been studied in \cite{Lame}, where it is shown to be equivalent 
to the Painlev\'e $(0,4)$ case with particular exponents, due to hyperellipticity of the curve.
Similarly, the case $(2,0)$ is studied in \cite{ViktoriaFrank} and turns to be the quotient of 
the Garnier case $(0,6)$ by an involution, again by hyperellipticity. The case $(1,2)$ involved in the present paper
is the first one that does not reduce to genus zero case: for generic eigenvalues $\nu$, the hyperelliptic involution
does not preserve the spectral data.

{\bf Results.}
We fix $C$ to be the elliptic curve with affine equation $y^2=x(x-1)(x-\lambda)$, $\lambda\in\mathbb C\setminus\{0,1\}$,
and denote by $w_\infty\in C$ the point at infinity.
In Section \ref{Sec:FlatCriterion}, the number $n=\deg(D)$ of poles and eigenvalues $\nu$ are arbitrary. There, we study 
which parabolic bundles $(E,{\bf p})$ over $(C,D)$ are $\nu$-flat, i.e. admit a connection $\nabla$ with prescribed
trace and exponents, compatible with parabolics. This has been done for a general curve in \cite{Bis}
in the orbifold case (i.e. with rational eigenvalues) and we extend his criterion for general eigenvalues
in Theorem \ref{criterion}. 

\begin{THM}
A parabolic bundle  $(E,\textbf{p})$ over an elliptic curve $(C,D)$ is $\nu$-flat
if, and only if, it satisfies 
$$\text{Fuchs relation:}\ \ \ \deg (E) + \sum_{k=1}^n (\nu_k^{+}+\nu_k^{-})=0$$
and for each decomposition $(E,\textbf{p})=(L',\textbf{p}')\oplus(L'',\textbf{p}'')$, we have 
$$\deg (L) + \sum_{p_k\in L}\nu_k^+ +\sum_{p_k\not\in L}\nu_k^-=0$$
for $L=L'$ and $L''$.
\end{THM}

For generic $\nu$ satisfying Fuchs relation (and $n>0$), the second identity
cannot occur and in that case we have:
$$\nu\text{-flat}\ \ \ \Leftrightarrow\ \ \ \text{indecomposable}.$$

For $n=2$, the moduli space of indecomposable bundles has been recently described in \cite{Nestor}.
It is a non Hausdorff scheme $X$ whose Hausdorff quotient is $\mathbb P^1\times\mathbb P^1$.
More precisely, there is an embedding $C\stackrel{\sim}{\to}\Gamma\subset \mathbb P^1\times\mathbb P^1$ 
as a bidegree $(2,2)$ curve, such that $X$ is the union of two copies $X_{<}$ and $X_{>}$ of $\mathbb P^1\times\mathbb P^1$ 
patched together outside the curve $\Gamma$. For instance, if we fix $\det(E)=\mathcal O_{C}( w_\infty)$,
then $X_{<}$ will correspond to those parabolic bundles $(E,\textbf{p})$ such that $E=E_1$ is the unique non trivial extension 
$$
0 \longrightarrow \mathcal O_C  \longrightarrow E_{1}  \longrightarrow \mathcal O_C(w)  \longrightarrow 0.
$$
Then, the missing $\nu$-flat bundles occurring in $X_{>}$ are those decomposable $E$ with indecomposable parabolic structure. More details are recalled in Section \ref{par}.

We denote by ${\rm Con}^{\nu}_{<}(C,D)$ the open subset of ${\rm Con}^{\nu}(C,D)$ over $X_{<}$.
In Section \ref{sec:universal}, an explicit universal family of connections is given for ${\rm Con}^{\nu}_{<}(C,D)$:
through a birational trivialization $E_1\dashrightarrow \mathcal O_C\times\mathcal O_C$,
it is given by an explicit family of Fuchsian system with $3$ additional apparent singular points.
To set our main result, given a parabolic bundle $(E,\textbf{p})$, let us introduce the parabolic structure $\textbf{p}^-=(p_1^-,p_2^-)$
associated to $\nu_k^-$-eigenspaces, and denote $\textbf{p}^+:=\textbf{p}$. Then, we have a natural map
$$\Par:\ \left\{\begin{matrix}{\rm Con}^{\nu}_{<}(C,D)&\to& X_{<}\times X_{<}\\
(E_1,\nabla)&\mapsto&((E_1,\textbf{p}^+),(E_1,\textbf{p}^-))
\end{matrix}\right.$$
Since $\nu$ is generic, we can assume $\nu_k^+\not=\nu_k^-$ for $k=1,2$, and therefore 
$p_k^+\not= p_k^-$. This implies that the image of $\Par$ in 
$X_{<}\times X_{<}$ avoid the ``incidence variety'' 
$$I:=\{z_1=\zeta_1\}\cup\{z_2=\zeta_2\}\subset \underbrace{(\mathbb P^1_{z_1}\times\mathbb P^1_{z_2})\times(\mathbb P^1_{\zeta_1}\times\mathbb P^1_{\zeta_2})}_{X_{<}\times X_{<}}.$$
Setting $\nu_k:=\nu_k^+-\nu_k^-$, our main result is (see Theorem \ref{isomorphism affine} and Section \ref{Sec:Symplectic})

\begin{THM}\label{MainTheorem}
If $\nu_1\cdot\nu_2\not=0$, then the map
$$\Par:\ {\rm Con}^{\nu}_{<}(C,D)\to X_{<}\times X_{<}$$
induces an isomorphism onto the complement of the incidence variety
$X_{<}\times X_{<}\setminus I$ and the image of the symplectic structure
is given by 
$$\omega=-\frac{1}{2}\left\{\nu_1\frac{dz_1\wedge d\zeta_1}{(z_1-\zeta_1)^2}+\nu_2\frac{dz_2\wedge d\zeta_2}{(z_2-\zeta_2)^2}\right\}.$$
\end{THM}

In the spirit of classical Torelli Theorem, this shows that exponents (difference of eigenvalues) 
can be read off from the moduli space, see Proposition \ref{prop:Torelli}. This completes the result of \cite[Theorem B]{Nestor}
where it is shown that the moduli space of $\nu$-flat bundles keeps track of the punctured curve $(C,D)$.

We investigate in Section \ref{sec:WholeModuli} how to cover the full moduli space ${\rm Con}^{\nu}(C,D)$
by three charts like the one in Theorem \ref{MainTheorem}, see Theorem \ref{main2}. Finally, in Section \ref{sec:Apparent} 
(see also Section \ref{Sec:Symplectic}), we study the ``apparent map'' which, to a connection $(E,\nabla)\in{\rm Con}^{\nu}_{<}(C,D)$,
 associates the position of the apparent singular points of the corresponding scalar equation via the ``cyclic
vector'' $\mathcal O_C\subset E_1$ (see \cite{IIS1,LS} for instance). This gives us a map
$$App\ :\ {\rm Con}^{\nu}_{<}(C,D)\to  |\mathcal O_C(w_{\infty}+t_1+t_2)|\simeq \mathbb P^2$$
which turns out to be Lagrangian. Similarly as \cite[Theorem 1.1]{LS} in genus zero case, we have 

\begin{THM}If $\nu_1+\nu_2+1\neq 0$, then the rational map 
$$Bun\times App\ :\ {\rm  {Con}}_<^{\nu}(C,D) \longrightarrow (\mathbb P^1_{z_1}\times \mathbb P^1_{z_2})\times \mathbb P^2$$
is birational. More precisely, the restriction
$$App\ :\ Bun^{-1}(\textbf{p})\longrightarrow \mathbb P^2$$
is injective if, and only if, $\textbf{p}\in X_{<}\simeq \mathbb P^1_{z_1}\times \mathbb P^1_{z_2}$ is lying outside 
 $\{z_1=t\}\cup\{z_2=t\}$, where $t$ is the first coordinate of $t_1$.
\end{THM}

{\bf Notation and conventions.} Curves are always assumed to be irreducible and defined over the field $\mathbb C$ of complex numbers.  Given a projective smooth curve $C$ and a holomorphic vector bundle $E \longrightarrow C$ over $C$  we make no difference in notation between the total space $E$ and its locally free sheaf defined by holomorphic sections. We denote by $E^{*}$ the sheaf $\mathcal Hom_{\mathcal O_C}(E,\mathcal O_C)$. If $E$ and $F$ are two vector bundles over $C$ we denote by ${\rm Hom}(E,F)$ the space of global sections of the sheaf $\mathcal Hom_{\mathcal O_X}(E,F)$. In particular, ${\rm End}(E)={\rm Hom}(E,E)$ denotes the set of endomorphisms of $E$. {\color{black}We use the notation $\mathbb P^1_z$ to denote 
the projective line $\mathbb P^1$ equipped with an affine coordinate $z\in\mathbb C$; this notation will be used to distinguish 
between several occurences of $\mathbb P^1$.}

\section{Existence of logarithmic connections}\label{Sec:FlatCriterion}

In this section, we shall investigate the  existence of logarithmic connections on a given rank 2 quasi-parabolic vector bundle $(E,\textbf{p})$ over an elliptic curve $C$. {\color{black}We give a criterion in Theorem \ref{criterion} which extends the famous Weil criterion 
\cite{Weil} for holomorphic connections, and Biswas criterion \cite{Bis} for logarithmic connections with rational residual 
eigenvalues.}

\subsection{Logarithmic connections}\label{logarithmic connections}

Let us fix a set of $n$ distinct points $\textbf{t}=\{t_1,...,t_n\}$ on a smooth projective complex curve $C$ 
and let $D=t_1+\cdots + t_n$ be  the reduced effective divisor associated to it. A \textit{logarithmic connection} on a vector bundle $E$ over $C$ with polar divisor $D$ is a $\mathbb C$-linear map
$$
\nabla: E \longrightarrow E\otimes \Omega_C^1(D)
$$
satisfying the Leibniz rule
$$
\nabla(f\cdot s)= df \otimes s + f \cdot \nabla (s)
$$
for any local section $s$ of $E$ and function $f$ on $C$. If $E$ is of rank $r$, then it is given by a cocycle $\{G_{ij}\}\in {\rm H}^1(C,GL(r,\mathcal O_{C}))$ defined by an atlas of trivializations $C=\cup U_i$ where $E|_{U_i}\simeq U_i\times \mathbb C^r$. Locally, over each $U_i$, the connection writes $\nabla = d_C + A_i$ where $d_C: \mathcal O_C \longrightarrow \Omega_C^1$ is the differential operator on $C$  and $A_i$ is a $r\times r$ matrix with coefficients in $\Omega_C^1(D)$ that glue together through the transition map. That is, $A_i$ is a $r\times r$ matrix of meromorphic $1$-forms having at most simple poles on $\textbf{t}$ and this collection of matrices must satisfies  
$$
A_j=G_{ij}\cdot A_i \cdot G_{ij}^{-1} + dG_{ij}\cdot G_{ij}^{-1}
$$  
over intersections $U_i\cap U_j$. For each pole $t_k\in U_i$, the residue homomorphism ${\rm Res}_{t_k}(\nabla)={\rm Res}_{t_k}(A_i) \in {\rm End}(E|_{t_k})$ is well defined. If $E$ is of rank $2$, then let $\nu_k^+$ and $\nu_k^-$ be the eigenvalues of   ${\rm Res}_{t_k}(\nabla)$, called \textit{local exponents} of $\nabla$ over $t_k$. The data 
$$
\nu=(\nu_1^{\pm},...,\nu_n^{\pm}) \in \mathbb C^{2n}
$$
is called \textit{local exponent} of $\nabla$.   We note that the connection $\nabla$ induces a logarithmic connection ${\rm tr}(\nabla)$ on the determinant line bundle ${\rm det}(E)$ with
$$
{\rm Res}_{t_k}({\rm tr}(\nabla))=\nu_k^+ + \nu_k^-.
$$
By  Residue Theorem  we obtain the Fuchs relation 
\begin{eqnarray}\label{Fuchs}
d + \sum_{k=1}^n (\nu_k^+ + \nu_k^-)=0.
\end{eqnarray}
where $d={\rm deg}(E)$.

When $E$ admits a holomorphic connection, then it is called \textit{flat}. There is a similar notion for quasi-parabolic vector bundles. We fix $(E, \textbf{\textbf{ p}}=\{p_k\})$ a quasi-parabolic rank $2$ vector bundle over $(C,\textbf{t})$, see Section \ref{def parabolic}. We shall say that  $(E, \textbf{\textbf{ p}})$ is $\nu$-\textit{flat} if it admits a logarithmic connection $\nabla$ with given local exponent $\nu$  satisfying  
$$
{\rm Res}_{t_k}(\nabla)(p_k) = \nu_k^+ \cdot p_k. 
$$
We also say that $(E, \nabla, \textbf{\textbf{ p}})$ is a $\nu$-\textit{parabolic connection}.

The following flatness criterion for vector bundles over curves is due to A. Weil (see \cite{Weil, At}):  a vector bundle is flat if, and only if, it is the direct sum of indecomposable vector bundles of degree zero.  A generalization to the parabolic context of the Weil's criterion was obtained in \cite{Bis} where the local exponents are supposed to be rational numbers. Besides that, when $C=\mathbb P^1$ and $\nu$ is generic
\begin{eqnarray}\label{generic}
\sum_{k=1}^n \nu_k^{\epsilon_k} \notin \mathbb Z \;;\;\;\; \epsilon_k \in \{+,- \}
\end{eqnarray}
the following equivalent conditions are known (cf. \cite{AL}):
\begin{enumerate}
\item $(E, \textbf{\textbf{ p}})$ is $\nu$-\textit{flat};
\item ${\rm End}_0(E, \textbf{\textbf{ p}}) = 0$;
\item $(E, \textbf{\textbf{ p}})$ is indecomposable.
\end{enumerate}

\begin{remark}\label{remark simple}
We note that if $C$ is an elliptic curve then the above conditions $(1)$ and $(2)$ are not equivalents. {\color{black}For instance, let $E_0$ be the unique non trivial extension
$$
0\longrightarrow \mathcal O_C {\longrightarrow}  E_0{\longrightarrow}  \mathcal O_C \longrightarrow 0.
$$}
If $\textbf{p}$ is a parabolic structure with all parabolics lying on $\mathcal O_C \hookrightarrow E_0$, then from Proposition \ref{case indecomposable}:
$$
{\rm End}_0(E_0,\textbf{p}) = \mathbb C.
$$ 
But $(E_0,\textbf{p})$ is indecomposable and, as we will see in Theorem \ref{criterion}, it is $\nu$-flat for any $\nu$ satisfying the Fuchs Relation.
\end{remark}

\subsection{Flatness criterion over elliptic curves}

In this section, we will obtain a parabolic version of Weil's criterion over elliptic curves, in the same spirit of \cite{AL, Bis} (see Theorem \ref{criterion}).

Before that, let us recall the definition of direct summand as well as parabolic degree. Let  $(E,\textbf{p})$ be a rank $2$ quasi-parabolic bundle. We say that  $(L,{\textbf{q}})$ is a \textit{direct summand} of $(E,\textbf{p})$ if either  $(E,\textbf{p})=(L,{\textbf{q}})$ or 
 $L$ is a line bundle and there exists another parabolic line bundle, say $(M,\textbf{r})$, such that 
$$
(E,\textbf{p}) = (L,{\textbf{q}})\oplus (M,\textbf{r}).
$$
Its \textit{parabolic degree}, denoted by $\deg_{\nu} (L,{\textbf{q}})$, with respect to  $\nu=( \nu_1^{\pm},...,\nu_n^{\pm})$ is defined as follows:
\begin{itemize}
\item  if   $(E,\textbf{p})=(L,{\textbf{q}})$ then we set
$$
{\rm deg}_{\nu} (L,{\textbf{q}}) :=  \deg (E) + \sum_{k=1}^n (\nu_k^{+}+\nu_k^{-});
$$
\item and if $(E,\textbf{p}) = (L,{\textbf{q}})\oplus (M,\textbf{r})$  then it is defined by
$$
{\rm deg}_{\nu} (L,{\textbf{q}}) :=  \deg (L) + \sum_{k=1}^n \nu_k^{\epsilon_k}
$$
where $\epsilon_k = +$ if $p_k$ is contained in $L$ and $\epsilon_k=-$ if $p_k$ is contained in $M$. 
\end{itemize}

\begin{thm}\label{criterion}
Given a quasi-parabolic bundle  $(E,\textbf{p})$ over an elliptic curve $C$, the following conditions are equivalents
\begin{enumerate}
\item  $(E,\textbf{p})$ is $\nu$-flat;
\item  every direct summand of $(E,\textbf{p})$ is of parabolic degree zero with respect to $\nu$.
\end{enumerate}
\end{thm}

The proof will be given in Section \ref{proof}. As a consequence, one obtains the following corollary for generic exponents $\nu$ satisfying the Fuchs Relation 
\begin{eqnarray*}
d + \sum_{k=1}^n (\nu_k^+ + \nu_k^-)=0.
\end{eqnarray*}

\begin{cor}\label{cor criterion}
Assume $\nu_1^{\epsilon_1}+\cdots+\nu_n^{\epsilon_n} \notin \mathbb Z$ for any  $\epsilon_k \in \{+,- \}$. Given a quasi-parabolic bundle  $(E,\textbf{p})$ over an elliptic curve $C$, with $\deg (E)=d$, the following conditions are equivalents
\begin{enumerate}
\item  $(E,\textbf{p})$ is $\nu$-flat;
\item  $(E,\textbf{p})$ is indecomposable.
\end{enumerate}
\end{cor} 

Recall that over $\mathbb P^1$, the above conditions $(1)$ and $(2)$ are equivalent to $(E,\textbf{p})$ be simple,
i.e. the only  automorphisms of $E$ preserving parabolics are scalar.  Here, over an elliptic curve it is no more true.  As in Remark \ref{remark simple}, if $\textbf{p}$ is a parabolic structure on $E_0$ with all parabolics lying on $\mathcal O_C \hookrightarrow E_0$ then ${\rm End}_0(E_0,\textbf{p}) = \mathbb C$. Hence, $(E_0,\textbf{p})$ is not simple:
\begin{eqnarray*}
{\rm Aut}(E_0,\textbf{p}) = \left\{
\left(
\begin{array}{ccc} 
a & b  \\
0 & a  \\
\end{array}
\right)
\;;\;\; a\in \mathbb C^*,\;\; b\in\mathbb C
\right\}.
\end{eqnarray*}

\subsubsection{Preliminary lemmas}

The easy part of Theorem \ref{criterion} is $(1)\Rightarrow(2)$. As we shall see in the following lemma, it is a consequence of Fuchs Relation. 

\begin{lemma}\label{restriction}
If $(E,\textbf{p})$ is $\nu$-flat, then every direct summand of $(E,\textbf{p})$ is of parabolic degree zero.
\end{lemma}

\proof
Let us suppose $(E,\textbf{p}) = (L,{\bf{q}}) \oplus  (M,{\bf{r}})$. The bundle $E$ is defined by gluing local charts $U_i\times \mathbb C^2$ with transition matrices 
\begin{eqnarray*}
M_{ij}=\left(
\begin{array}{ccc} 
a_{ij} & 0  \\
0 & b_{ij}  \\
\end{array}
\right)
\end{eqnarray*}
where the subbundles $L$ and $M$ are respectively generated by $e_1$ and $e_2$.

Let $\nabla$ be a connection over $(E,\textbf{p})$ with local exponent $\nu$, given in those charts  by $\nabla = d + A_i$ where 
\begin{eqnarray*}
A_i=\left(
\begin{array}{ccc} 
\alpha_i & \beta_i  \\
\gamma_i & \delta_i  \\
\end{array}
\right)
\end{eqnarray*}

First we note that $\nabla_1:= \nabla|_{L} = d + \alpha_i$ defines a connection over $L$. In fact,  the compatibility conditions for $\nabla$ imply 
$$
\alpha_j = \alpha_i + a_{ij}^{-1}\cdot da_{ij}.
$$
Then under hypothesis of $(E,\textbf{p})$ be decomposable one obtains 
$$
{\rm Res}_{t_k}(\nabla_1) = \nu_k^{\epsilon_k}
$$
where $\epsilon_k = +$ if $p_k$ lies in $L$ and $\epsilon_k = -$ if $p_k$ lies in $M$. The conclusion of the proof follows from Fuchs Relation of $\nabla_1$ over $L$.
\endproof

In order to prove that $(2)$ implies $(1)$ in Theorem \ref{criterion}, let us consider the set ${\rm End}_0(E,\textbf{p})$ of traceless endomorphisms  of $E$ leaving fixed the parabolics. In the context of $\mathfrak {sl}_2$-connections, the vanishing of this set is a sufficient condition to a quasi-parabolic vector bundle be $\nu$-flat. This has been proved in \cite[Proposition 3]{AL} in the case $X=\mathbb P^1$, but the same proof works to the general case. Here  we do the details for the reader's convenience. 

\begin{lemma}\label{cohomology}
Let $C$ be a projective smooth curve of genus $g\ge 0$ and $E$ be a rank $2$ vector bundle over $C$ with trivial determinant bundle. If 
$$
{\rm End}_0(E,\textbf{p}) = \{0\}
$$
then $(E,\textbf{p})$ is $\nu$-flat for any $\nu$ satisfying the Fuchs Relation. 
\end{lemma}

\proof
To give a logarithmic connection $\nabla$ on $(E,\textbf{p})$ with local exponent $\nu$ it is equivalent to give an $\mathfrak {sl}_2$-connection $\nabla \otimes \zeta$ on $(E,\textbf{p})$ for  suitable connection $\zeta$ on $\mathcal O_C$. Then let us assume $\nu = (\pm\nu_1,...,\pm\nu_n)$. 

We must construct a logarithmic connection $\nabla$ over  $(E,\textbf{p})$ with local exponent $\nu$ satisfying
\begin{eqnarray}\label{kernel}
{\rm Res}_{t_k}(\nabla)(p_k) = \nu_k \cdot p_k
\end{eqnarray}
for any $k=1,...,n$.  We can do it locally on $C$. The obstruction to global existence is measured by a certain cohomology group. In fact, let $\nabla_i$ be a connection over $E|_{U_i}$ satisfying the desired condition
for any $k=1,...,n$ such that $t_k \in U_i$ for suitable open sets $U_i\subset C$ covering $C$. Note that $\theta_{ij}=\nabla_i - \nabla_j$ satisfies $ {\rm Res}_{t_k}(\theta_{ij})(p_k)=0$, by (\ref{kernel}). Then the differences $\nabla_i - \nabla_j$ define a cocycle
$$
\{\nabla_i - \nabla_j\} \in {\rm H}^1(C, \mathcal E)
$$
where 
$$
\mathcal E = \{ \theta \in \mathcal End_0 (E)\otimes \Omega_C^1(D) \;; \; {\rm Res}_{t_k}(\theta)(p_k)=0\}.
$$

Hence the global existence is insured by the vanishing of this cohomology group. 
Here ${\mathcal End}_0(E)$ is the sheaf of traceless endomorphisms.
We will prove that there is an isomorphism
$$
{\rm H}^1(C, \mathcal E) \simeq {\rm End}_0(E,\textbf{p})^*.
$$
In order to prove it, let us remark that, from Serre's duality theorem, ${\rm H}^1(C, \mathcal E)$ is dual to ${\rm H}^0(C,\mathcal E^*\otimes \Omega_C^1)$. Let us consider   the $\mathcal O_C$-bilinear symmetric map 
\begin{eqnarray*}
	{\mathcal End}_0(E) \times {\mathcal End}_0(E) &\longrightarrow& \mathcal O_C.\\
		(A,B)                    &\mapsto&      {\rm tr} (A\cdot B)
\end{eqnarray*}
It induces an isomorphism $\psi: {\mathcal End}_0(E) \longrightarrow {\mathcal End}_0(E)^*$. Let $\tilde{\mathcal E}$ be the subsheaf of ${\mathcal End}_0(E)$ defined by 
$$
\tilde{\mathcal E} = \{\theta \in {\mathcal End}_0(E) \;;\;\; \theta (p_k) = 0\}.
$$
We note that $\mathcal E$ coincides with $\tilde{\mathcal E} \otimes \Omega_C^1(D)$. Since the image of $\tilde{\mathcal E}$ by  $\psi$ is $\tilde{\mathcal E}^*$  one obtains
$$
\tilde{\mathcal E}^* \otimes \mathcal O_C(-D) = \{ \alpha \in {\mathcal End}_0(E)  \;;\;\; \alpha (p_k) \subset p_k \}.
$$
This implies that
$$
{\rm H}^0(C,\tilde{\mathcal E}^* \otimes \mathcal O_C(-D)) = {\rm End}_0(E,\textbf{p}).
$$
and we get the desired isomorphism:
$$
{\rm H}^1(C, \mathcal E) \simeq {\rm End}_0(E,\textbf{p})^*.
$$
\endproof

Now we shall give the proof of Theorem \ref{criterion}.

\subsubsection{Proof of Theorem \ref{criterion}}\label{proof}

It follows from Lemma \ref{restriction} that $(1)$ implies $(2)$. Now let  $(E,\textbf{p})$ be a quasi-parabolic bundle satisfying the hypothesis $(2)$ of the statement.  First, let us assume that $(E,\textbf{p})$ is decomposable:
$$
(E,\textbf{p}) = (L_1,{\textbf{p}_1})\oplus (L_2,\textbf{p}_2).
$$
Since $(L_i,{\textbf{p}_i})$ has parabolic degree zero for each $i\in \{1,2\}$, there is a logarithmic connection $\alpha_i$  on $L_i$, satisfying 
$$
{\rm Res}_{t_k}(\alpha_i) = \nu_k^{\epsilon_k}
$$
where $\epsilon_k=+$ if $p_k$ lies in $L_i$ and $\epsilon_k=-$ if $p_k$ does not lie in $L_i$. They define a diagonal connection 
\begin{eqnarray*}
\nabla=\left(
\begin{array}{ccc} 
\alpha_1 & 0  \\
0 & \alpha_2  \\
\end{array}
\right)
\end{eqnarray*}
on $(E,\textbf{p})$ with local exponent $\nu$.

Now let $(E,\textbf{p})$ be indecomposable.
After elementary transformations and twists, one can assume that $E$ has trivial determinant bundle. Besides that, it is enough to show the existence of  an $\mathfrak {sl}_2$-connection on $(E,\textbf{p})$ with local exponent $\nu=(\pm\nu_1,...,\pm\nu_n)$. It follows from Proposition \ref{totheorem} that 
$$
{\rm End}_0(E,\textbf{p}) = \{0\}
$$
unless $E$ equals $E_0$, up to elementary transformations and twists, and all the parabolics lie in the maximal subbundle $\mathcal O_C \hookrightarrow E_0$. Then, from Lemma \ref{cohomology}, to conclude the proof of Theorem \ref{criterion}, it is enough to consider the case $E_0$ with all the parabolics in  $\mathcal O_C \hookrightarrow E_0$. 
We shall do it in the following lemma. 

\begin{lemma}\label{existence sl2}
Let $(E_0,\textbf{p})$ be a quasi-parabolic bundle where all the parabolics lie in the maximal subbundle $\mathcal O_C\hookrightarrow E_0$. Then there exists an $\mathfrak {sl}_2$-connection on $(E_0,\textbf{p})$ with given local exponent $(\pm\nu_1,...,\pm\nu_n)$.
\end{lemma}

\proof
The idea of the proof is the following: $E_0$ can be obtained from $\mathcal O_C\oplus \mathcal O_C(t_1)$ by one negative elementary transformation at a direction $q\subset \mathcal O_C\oplus \mathcal O_C(t_1)|_{t_1}$ which is not contained in $\mathcal O_C|_{t_1}$ neither in $\mathcal O_C(t_{1})|_{t_1}$. The maximal subbundle $\mathcal O_C \hookrightarrow E_0$ corresponds to the maximal subbundle $\mathcal O_C(t_1) \hookrightarrow \mathcal O_C\oplus \mathcal O_C(t_1)$. Then we need to construct a logarithmic connection $\nabla$ on $\mathcal O_C\oplus \mathcal O_C(t_1)$ with poles at $D=t_1+\cdots +t_n$, satisfying the following conditions:
\begin{enumerate}
\item ${\rm Res}_{t_1}(\nabla)$ has $-\nu_1$ and $\nu_1 - 1$ as eigenvalues, where   $-\nu_1$  is associated with the eigenspace $q$ as above;
\item ${\rm Res}_{t_k}(\nabla)$, $k= 2,...,n$, has $\pm \nu_k$ as eigenvalues, where  $\nu_k$ is associated with the eigenspace  $\mathcal O_C(t_{1})|_{t_k}$.
\end{enumerate}

We can assume that the vector bundle $\mathcal O_C\oplus \mathcal O_C(t_1)$ is defined by the cocycle 
\begin{eqnarray*}
G_{ij}=\left(
\begin{array}{ccc} 
1 & 0  \\
0 & a_{ij}  \\
\end{array}
\right)
\end{eqnarray*}
where $\{a_{ij}\}$ defines the line bundle $\mathcal O_C(t_1)$.

Now, to give a logarithmic connection  $\nabla$ on $\mathcal O_C\oplus \mathcal O_C(t_1)$ with poles at $D$ is equivalent to give $\nabla = d+A_i$ in charts $U_i\subset C$ where 
\begin{eqnarray*}
A_i=\left(
\begin{array}{ccc} 
\alpha_i & \beta_i  \\
\gamma_i & \delta_i  \\
\end{array}
\right)
\in {\rm GL}_2(\Omega_{U_i}^1(D))
\end{eqnarray*}
with $\alpha=\{\alpha_i\}$, $\beta=\{\beta_i\}$, $\gamma=\{\gamma_i\}$ and $\delta=\{\delta_i\}$ satisfying the compatibility conditions:
$$
A_i\cdot G_{ij}=G_{ij}\cdot A_j + dG_{ij}\cdot G_{ij}^{-1}
$$
on each intersection $U_i\cap U_j$. Equivalently, 
\begin{displaymath}
\left\{ \begin{array}{ll}
d+\alpha & \textrm{is a connection on $\mathcal O_C$}\\
d + \delta & \textrm{is a connection on $\mathcal O_C(t_1)$}\\
\beta   & \textrm{defines a global section  $\{a_i \beta_i\}\in{\rm H}^0(C,\Omega_C^1(D-t_1))$}\\
\gamma   & \textrm{defines a global section $\{a_i^{-1} \gamma_i\}\in{\rm H}^0(C,\Omega_C^1(D+t_1))$}
\end{array} \right.
\end{displaymath}
where $a_{ij}=\frac{a_i}{a_j}$ is a meromorphic resolution of the cocycle, the vector $e_1$ generates the subbundle $\mathcal O_C$ and $e_2$ generates $\mathcal O_C(t_1)$ on each chart $U_i\subset C$.

In order to find $\alpha$ and $\delta$ we define $\lambda^+$ and $\lambda^-$ as
\begin{eqnarray*}
\lambda^+ &=& \nu_2+\cdots +\nu_n\\  
\lambda^- &=& -\nu_2-\cdots - \nu_n -1.
\end{eqnarray*}
Let $d+\alpha$ be a logarithmic connection on $\mathcal O_C$ satisfying 
$$
{\rm Res}_{t_1}(\alpha)=\lambda^+\;; \;\;{\rm Res}_{t_k}(\alpha)=-\nu_k\;, k\ge 2
$$
and $d+\delta$ be a logarithmic connection on $\mathcal O_C(t_1)$ satisfying  
$$
{\rm Res}_{t_1}(\delta)=\lambda^-\;; \;\;{\rm Res}_{t_k}(\delta)=\nu_k\;, k\ge 2.
$$

It remains to find $\beta$ and $\gamma$ subject to the conditions $(1)$ and $(2)$ above. Since residues of $\alpha$ and $\delta$ have already been chosen,  condition $(2)$  is equivalent to the vanishing of $\beta$ in $t_k$ for $k\ge 2$, that is, $\beta$ is induced by an element
$$
\beta' \in {\rm H}^0(C,\Omega_C^1(D-t_1)\otimes \mathcal O_C(-t_2-\cdots -t_n))= {\rm H}^0(C,\Omega_C^1).
$$
In addition, as $C$ is an elliptic curve, we can assume $\beta'\in \mathbb C$. Then we have one degree of freedom to choose $\beta$ satisfying condition $(2)$, which is determined by its residue at $t_1$, still denoted by $\beta' = {\rm Res}_{t_1}(\beta)$. 

The choice of $\gamma$ is independent of condition $(2)$. For instance for each $k\ge 2$, ${\rm Res}_{t_k}(\gamma)$ determines the eigenvector associated with the eigenvalue $-\nu_k$.   But we have to choose $\gamma$ satisfying condition $(1)$. Given $\gamma \in {\rm H}^0(C,\Omega_C^1(D+t_1))$  with residue $\gamma'$ at $t_1$, one gets
\begin{eqnarray*}
{\rm Res}_{t_1}(\nabla)=\left(
\begin{array}{ccc} 
\lambda^+ & \beta'  \\
\gamma' & \lambda^-  \\
\end{array}
\right).
\end{eqnarray*}
Then saying that $-\nu_1$ and $\nu_1 - 1$ are eigenvalues of ${\rm Res}_{t_1}(\nabla)$ is equivalent to  
$$
\lambda^+ \lambda^- - \beta'\gamma' = (-\nu_1)(\nu_1-1).
$$
We leave the reader to verify that this last equality is equivalent to 
$$
\beta'\gamma'  = (\nu_1+\lambda^+)\cdot (\nu_1-\lambda^+-1).
$$

On one hand, if the term in the right side is non zero, we can find $\beta'\neq 0$ and $\gamma'\neq 0$ satisfying this equality. In fact,  $\gamma'\neq 0$ if and only if $\gamma$ does not vanish at $t_1$. Then we need to find $\gamma\in{\rm H}^0(C,\Omega_C^1(D+t_1))$ which is not contained in  ${\rm H}^0(C,\Omega_C^1(D))$. We can do it, because $C$ is elliptic, hence the linear system which corresponds to ${\rm H}^0(C,\Omega_C^1(D+t_1))$ has no base points. On the other hand, if the term in the right side equals zero, we set 
\begin{displaymath}
\left\{ \begin{array}{ll}
\beta' = 0, \gamma'\neq 0 & \textrm{if $\nu_1+\lambda^+ =0$}\\
\beta'\neq 0, \gamma'=0  & \textrm{if $\nu_1-\lambda^+-1=0$}
\end{array} \right.
\end{displaymath}
Finally, with these choices we see that the eigenspace  associated to $-\nu_1$ is not contained in $\mathcal O_C|_{t_1}$ neither in $\mathcal O_C(t_{1})|_{t_1}$.
\endproof

\subsection{Indecomposable quasi-parabolic bundles} In view of Corollary \ref{cor criterion} it is interesting to characterize degree one indecomposable rank $2$ quasi-parabolic bundles when there is at least one parabolic point.
Fix $\det(E)=\mathcal O_C(w)$, $w\in C$.

\begin{prop}\label{prop decomposable}
Assume $n\ge 1$ and let $L$ be a degree $k$ line bundle over the elliptic curve $C$. 
If the parabolic bundle $(L\oplus L^{-1}(w),\textbf{p})$ is indecomposable then
$$
-n+1 \le 2k \le n+1.
$$
\end{prop}

\proof
Firstly let us assume $2k < -n+1$. Since $C$ is elliptic, it follows from Riemann-Roch theorem that
$$
\dim {\rm{H}}^0(C,\mathcal O \oplus L^{-2}(w)) = -2k+2.
$$
Then the family of embeddings $L\hookrightarrow L\oplus L^{-1}(w)$ is of dimension $-2k+1$. Our hypothesis on $k$ implies that we can choose an embedding of $L$ containing any parabolic lying outside $L^{-1}(w)$. In this case our quasi-parabolic bundle would be decomposable.
The same argument works with $L^{-1}(w)$ instead of $L$ if $2k > n+1$.  
\endproof

The case $n=2$ is of particular interest for us. Then we shall give a simple consequence of this proposition that will be useful in the sequel. The following proposition characterizes quasi-parabolic bundles $(E,\textbf{p})$, $\textbf{p}=\{ p_1,p_2\}$, of rank $2$ and determinant $\mathcal O_C(w)$ arising in our moduli space of connections. 

\begin{prop}\label{prop char}
Let $(E,\textbf{p})$, $\textbf{p}=\{ p_1,p_2\}$, be a rank $2$ quasi-parabolic bundle over an elliptic curve $C$ with $\det (E) = \mathcal O_C (w)$. If  $(E,\textbf{p})$ is indecomposable then one of the following assertions hold true
\begin{enumerate}
\item $E\simeq L\oplus L^{-1}(w)$ with $\deg (L) = 0$; {\color{black}moreover, the parabolic structure satifies:
\begin{itemize}
\item $p_1,p_2$ do not lie neither on $L^{-1}(w)$, and nor on the same embedding $L\hookrightarrow E$
(a codimension one condition);
\item $L^2=\mathcal O_C(w-t_k)$ for some $k=1,2$, and $p_k$ lie outside of $L$ and $L^{-1}(w)$;
\end{itemize} }
\item $E=E_{1}$ is indecomposable, defined by the unique non trivial extension 
$$
0 \longrightarrow \mathcal O_C  \longrightarrow E_{1}  \longrightarrow \mathcal O_C(w)  \longrightarrow 0.
$$
\end{enumerate}
\end{prop}

\proof
If $E$ is decomposable with $\det (E) = \mathcal O_C(w)$ then Proposition \ref{prop decomposable} implies $E\simeq L\oplus L^{-1}(w)$ with $\deg (L)$ equals $0$ or $1$. If $\deg (L)=0$ we are done. If $\deg (L)=1$, then we set $M:=L^{-1}(w)$ to get $E\simeq M^{-1}(w)\oplus M$. Since $\deg (M)=0$, this gives the first assertion when $E$ is decomposable. 
{\color{black} Now any decomposition of $E$ is given by the unique destabilizing subsheaf $L^{-1}(w)$ together with 
any embedding $L\hookrightarrow E$. Such an embedding is given by a linear combination of the initial factor $L$
with a global section of $\mathrm{Hom}(L,L^{-1}(w))\simeq L^{-2}(w)$, so that we get a one parameter family of possible decompositions. We note that all embedings coincide with $L$ exactly at $t\in C$ where $L^{-2}(w)=\mathcal O_C(t)$.
If $t\not=t_1$, then there is a unique $L\subset E$ passing through $p_1$, provided that $p_1\not\subset L^{-1}(w)$;
then a generic $p_2$, i.e. not belonging to either that $L$ or $L^{-1}(w)$ will be an obstruction to decomposablity.
On the other hand, if $t=t_1$, and if $p_1$ is generic, then there is no $L$ passing through and $(E,\textbf{p})$ is indecomposable.
After reasoning similarly at $t_2$, one easily deduce that the only cases where we cannot find a decomposition of $E$ compatible with parabolics
are those two cases listed in the statement.}

If $E$ is indecomposable and $\det (E) = \mathcal O_C(w)$, then its well known that $E$ contains $\mathcal O_C$ as maximal subbundle. Moreover, there is only one indecomposable rank 2 bundle, up to isomorphism, having $\mathcal O_C(w)$ as determinant bundle (see for example \cite[Lemma 4.4 and 4.5]{Ma1}).
\endproof

\section{Moduli space of connections on elliptic curves with two poles}

Now we shall fix the data for our moduli space of logarithmic  connections with two poles over an elliptic curve; see Section \ref{moduli space} for a more complete introduction to moduli spaces of connections.  Let $C$ be an elliptic curve, for computation we assume that $C\subset \mathbb P^2$  is the smooth projective cubic curve
\begin{eqnarray}\label{elliptic}
zy^2 = x(x-z)(x-\lambda z)
\end{eqnarray}
with $\lambda \in \mathbb C$, $\lambda \neq 0,1$. Let us denote by $w_{\infty} = (0:1:0) \in C$ the identity 
with respect to the group structure, and 
\begin{eqnarray}\label{torsionpoints}
w_{0}=(0:0:1),\ w_1=(1:0:1),\ w_\lambda=(\lambda:0:1)
\end{eqnarray}
the $2$-torsion points. 

Let $D=t_1 + t_2$ be a reduced divisor on $C$, where $t_2=-t_1$  with respect to the group structure of $C$, 
i.e. defined by say $x=t$.  We assume $t_1\not=t_2$, i.e. $x\not=0,1,\lambda,\infty$. Let us fix local exponents
$\nu = (\nu_1, \nu_2) \in\mathbb C^{2}$ 
and define eigenvalues 
\begin{eqnarray}\label{nu}
\nu = (\nu_1^{\pm},\nu_2^{\pm}):=\left (\pm \frac{\nu_{1}}{2} - \frac{1}{2}, \pm \frac{\nu_{2}}{2}\right ). 
\end{eqnarray}
therefore satisfying Fuchs Relation
\begin{eqnarray*}
\nu_1^+ + \nu_1^-+\nu_2^+ + \nu_2^-+1=0.
\end{eqnarray*}
To avoid dealing with reducible connections, we assume moreover a generic condition 
$$
\nu_1\pm \nu_2\not\in \mathbb Z\setminus2\mathbb Z\ \ \ (\text{i.e. not an odd integer})
$$
so that $\nu_1^{\epsilon_1}+\nu_2^{\epsilon_2} \notin \mathbb Z$
for any  $\epsilon_k \in \{+,- \}$.

Fix $\zeta: \mathcal O_C(w_{\infty}) \longrightarrow \mathcal O_C(w_{\infty})\otimes \Omega^1_C(t_1)$ any rank one logarithmic connection on $\mathcal O_C(w_{\infty})$ satisfying 
$$
{\rm Res}_{t_1}(\zeta)=-1
$$
Since local $\zeta$-horizontal sections have a simple zero at $t_1$, the invertible sheaf generated by these sections
is $\mathcal O_C(w_{\infty}-t_1)$,
and $\zeta$ corresponds to a holomorphic connection on this latter bundle. In particular, the monodromy of $\zeta$
must be non trivial.

We denote by ${\rm Con}^{\nu}(C,D)$ the moduli space of triples $(E,\nabla, \textbf{p})$ where 
\begin{enumerate}
\item $(E, \textbf{p})$ is a rank $2$ quasi-parabolic vector bundle over $(C,D)$ having $\mathcal O_C(w_{\infty})$ as determinant bundle;
\item $\nabla: E \longrightarrow E\otimes \Omega_C^1(D)$ is a logarithmic connection on $E$ with polar divisor $D$, having $\nu$ as local exponents, ${\rm Res}_{t_k}(\nabla)$ acts on $p_k$ by multiplication by $\nu_k^+$ and 
its trace is given by ${\rm tr}(\nabla) = \zeta$;
\item two triples $(E,\nabla, \textbf{p})$ and $(E',\nabla',\textbf{p}')$ are equivalent when there is an isomorphism between quasi-parabolic bundles $(E,\textbf{p})$ and $(E',\textbf{p}')$ conjugating $\nabla$ and $\nabla'$. 
\end{enumerate}

The reason why we have chosen $\mathcal O_C(w_{\infty})$ instead of $\mathcal O_C$ as our fixed determinant bundle is the following:  there is an open subset of ${\rm Con}^{\nu}(C,D)$ formed by triples $(E_1,\nabla, \textbf{p})$ where $E_1$ is the unique indecomposable rank $2$ bundle over $C$, up to isomorphism, corresponding to the extension 
$$
0 \longrightarrow \mathcal O_C  \longrightarrow E_1  \longrightarrow \mathcal O_C(w_{\infty})  \longrightarrow 0.
$$
This allow us to avoid dealing with varrying underlying vector bundle. 
Some consequences of this choice will be more clear in the next section. 

In the next lines, we will study the forgetful map $Bun: (E,\nabla, \textbf{p}) \mapsto  (E, \textbf{p})$ from  ${\rm Con}^{\nu}(C,D)$  to the moduli space ${\rm Bun}(C,D)$ of \textit{parabolic bundles} $ (E, \textbf{p})$ having $\mathcal O_C(w_{\infty})$ as determinant bundle. Actually,  the construction of the moduli space ${\rm Con}^{\nu}(C,D)$ needs a choice of weights to impose a stability condition. But under our generic hypothesis on $\nu$, all connections are stable and the construction does not depend of this choice.  On the other hand, to obtain a good moduli space of quasi-parabolic bundles, we need to introduce a stability condition  (see Section \ref{def parabolic}).

\subsection{Moduli space of parabolic vector bundles}\label{par} 
In this section we recall the construction of the moduli space of parabolic bundles over an elliptic curve, see \cite{Nestor} for details.  Given weights  $\mu=(\mu_1, \mu_2) \in [0,1]^2$, let ${\rm Bun}_{w_{\infty}}^{\mu}(C,D)$ be the moduli space of $\mu$-semistable parabolic bundles with fixed determinant bundle $\mathcal O_C(w_{\infty})$. The space of weights is divided in two chambers by a wall $\mu_{1}+\mu_{2}=1$. Inside each chamber any point in ${\rm Bun}_{w_{\infty}}^{\mu}(C,D)$ is represented by the same bundle and all $\mu$-semistable bundles are $\mu$-stable. Strictly $\mu$-semistable bundles only occurs along the wall. It follows from \cite[Theorem A]{Nestor} that  ${\rm Bun}_{w_{\infty}}^{\mu}(C,D)$ is isomorphic to $\mathbb P^1\times \mathbb P^1$ for any choice of weights. We recall briefly this construction in the next lines to the readers convenience. 

Firstly, we describe parabolic bundles inside the chambers, see also \cite[Proposition 4.4]{Nestor}. Let us denote by $X_<={\rm Bun}_{w_{\infty}}^<(C,D)$ the moduli space corresponding to the chamber $0<\mu_1+\mu_2< 1$ and by $X_>={\rm Bun}_{w_{\infty}}^>(C,D)$ the other moduli space corresponding to  $1<\mu_1+\mu_2< 2$. All the $\mu$-semistable parabolic bundles arising in $X_<$ are of the form $(E_1,\textbf{p}=\{p_1,p_2\})$. Maximal subbundles of $E_1$ have degree zero, then any parabolic bundle in $X_<$ is $\mu$-stable. Each parabolic bundle is completely determined by 
$$
(p_1,p_2) \in E_1\vert_{t_1}\times E_1\vert_{t_2}\simeq\mathbb P^1\times \mathbb P^1. 
$$
Then, we get the identification $X_< \simeq \mathbb P^1\times \mathbb P^1$.  When $1<\mu<2$,  parabolic bundles $(E_1,\{p_1,p_2\})$ having a degree zero line bundle $L   \hookrightarrow E_1$ passing through the two parabolic directions became unstable, where $L$ is destabilizing:
$$
{\rm Stab}(L) = 1-\mu_{1}-\mu_{2}<0.
$$
But we need to add parabolic bundles of the following form 
$$
(L\oplus L^{-1}(w_{\infty}),\{p_1,p_2\})\;,\;\; \deg(L)=0
$$ 
where no parabolic $p_k$ lie on $L^{-1}(w_{\infty})$ and not all $p_k$ lie on the same embedding $L\hookrightarrow E$
(see Proposition \ref{prop char}). Since automorphisms group of $L\oplus L^{-1}(w_{\infty})$ is two dimensional, all parabolic bundles $$(L\oplus L^{-1}(w_{\infty}),\{p_1,p_2\})$$ with $L$ fixed, represent the same element in $X_>$. Then one obtains an identification $X_> \simeq \mathbb P^1\times \mathbb P^1$.  

When $\mu$ is inside the wall, the picture is described in \cite[Proposition 4.1 and Theorem 4.2]{Nestor}. If $\mu_{1}+\mu_{2}=1$ with $\mu_{k}\neq 0$ for $k=1,2$,  respective parabolic bundles in  $\Gamma_{<}$ and $\Gamma_{>}$ are identified and parabolic bundles $(L\oplus L^{-1}(w_{\infty}),\textbf{p})$ with both direction on $L$ appear. But if $\mu_{k}=0$, then we also find bundles $L\oplus L^{-1}(w_{\infty})$ with $p_{k}$ lying on $L^{-1}(w_{\infty})$.

Let us identify $C$ with its Jacobian variety ${\rm Jac}(C)$ of degree zero line bundles
\begin{eqnarray*}
C &\longrightarrow& {\rm Jac}(C). \\
    p  &\mapsto& \mathcal O_C(w_{\infty}-p)
\end{eqnarray*}
The locus $\Gamma$ of parabolic bundles inside $X_>$ that became unstable when we change the chamber are parametrized by $C$
\begin{eqnarray*}
C\simeq {\rm Jac}(C) &\hookrightarrow& \Gamma\subset X_<\\
    L  &\mapsto& (E,\textbf{p}_L)
\end{eqnarray*}
where $E=E_1$ and $\textbf{p}_L=\{p_1^L,p_2^L\}$ corresponds to the parabolic directions inside $L$. Similarly, we still call $\Gamma$ the respective locus inside $X_>$ parametrized by the following map
\begin{eqnarray*}
C\simeq {\rm Jac}(C) &\hookrightarrow& \Gamma\subset X_>\\
    L  &\mapsto& (E,\textbf{p}_L)
\end{eqnarray*}
where $E=L\oplus L^{-1}(w_{\infty})$, one of the parabolics lying in $L$ and the other one outside both $L$ and $L^{-1}(w_{\infty})$.
This locus $\Gamma$ corresponds to a curve of type $(2,2)$ parametrized by $C$. 

The two charts $X_<$ and $X_>$ are identified outside  $\Gamma$ and this provides a stratification of the coarse moduli space of \textit{simple parabolic bundles}
\begin{eqnarray}\label{strat}
{\rm Bun}_{w_{\infty}}(C,D) = X_<\sqcup X_>.
\end{eqnarray}
That is, simple parabolic bundles are parametrized by two copies of $\mathbb P^{1}\times \mathbb P^{1}$ identified outside a $(2,2)$ curve $\Gamma$. Now we would like to characterize indecomposable but not simple parabolic bundles, see also \cite[Table 1]{Nestor}.

\begin{prop}\label{not simple}
Let $(E, \textbf{p})$ be an indecomposable parabolic bundle over $(C,D)$ having $\mathcal O_{C}(w_{\infty})$ as determinant line bundle. If $(E, \textbf{p})$ is not simple, then 
$$
(E,\textbf{p}) = (L\oplus L^{-1}(w_{\infty}),\textbf{p})
$$ 
where $L^2 = \mathcal O_C(w_{\infty}-t_k)$, $k\in\{1,2\}$. Moreover, $p_k$ does not lie in $L$ neither in $L^{-1}(w_{\infty})$ and the other parabolic lies in $L^{-1}(w_{\infty})$.
\end{prop}

\proof
If $(E,\textbf{p})$ is indecomposable but not simple, then it follows from Proposition  \ref{prop char} that 
$$
(E,\textbf{p}) = (L\oplus L^{-1}(w_{\infty}),\textbf{p})\;,\;\; \deg(L)=0.
$$ 
Since $(E,\textbf{p})$ is not simple, at least one of the parabolics lies in $L^{-1}(w_{\infty})$.
Writing  
$$
L^2=\mathcal O_C(w_{\infty}-w)
$$
for some $w\in C$, one obtains $$\mathbb P E=\mathbb P (\mathcal O_C\oplus\mathcal O_C(w)).$$ 
The family of sections corresponding to the family of embeddings 
$L\hookrightarrow E$ has a base point at the fiber over $w\in C$. Therefore if $w$ is distinct of $t_k$, $k=1,2$,  
we can choose an $L$ containing the parabolic outside $L^{-1}(w_{\infty})$. In this situation, $(E,\textbf{p})$ is decomposable. Now if $w=t_k$, then $p_k$ must be outside both $L$ and  $L^{-1}(w_{\infty})$, and the other parabolic must lie in  $L^{-1}(w_{\infty})$; indeed, otherwise $(E,\textbf{p})$ would be decomposable (see Figure \ref{basepoint}).
\endproof

\begin{center}
\begin{figure}[h]
\centering
\includegraphics[height=1.3in]{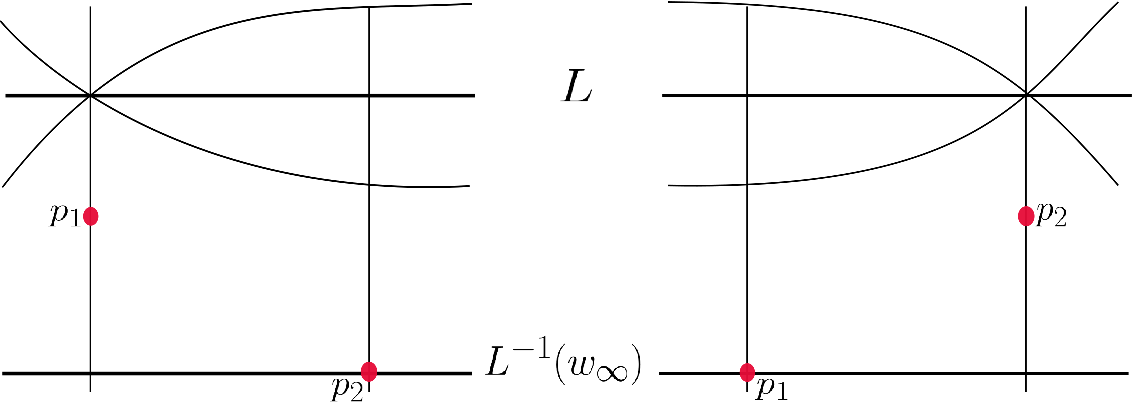}
\caption{Indecomposable not simple, $L^2 = \mathcal O_C(w_{\infty}-t_k)$, $k\in\{1,2\}$.}
\label{basepoint}
\end{figure}
\end{center}

\subsection{Fiber compactification of the moduli space and Higgs fields} 

Let us consider the following open subset of the moduli space of connections corresponding to 
{\color{black}stable} parabolic bundles in $X_<$: 
$$
{\rm Con}_<^{\nu}(C,D)= \{ (E,\nabla, \textbf{p}) \in {\rm Con}^{\nu}(C,D)\;;\;\; (E, \textbf{p}) \in  X_<\}.
$$
Similarly,  we define ${\rm Con}_>^{\nu}(C,D)$ with $X_>$ instead of $X_<$. We note that ${\rm Con}_<^{\nu}(C,D)$ corresponds to triples $(E,\nabla, \textbf{p})$ where $E=E_{1}$.

The union ${\rm Con}_<^{\nu}(C,D)\cup {\rm Con}_>^{\nu}(C,D)$ is the locus where the forgetful map with values at simple parabolic bundles  ${\rm Con}^{\nu}(C,D) \longrightarrow {\rm Bun}_{w_{\infty}}(C,D)$ is well defined.

Given a parabolic bundle  $(E_1,\textbf{p})\in X_<$, any two connections $\nabla$ and $\nabla'$ on it differ to each other by a parabolic Higgs field
$$
\nabla' - \nabla = \Theta \in {\rm H}^0(C,{\rm End}(E_1,\textbf{p})\otimes 	\Omega_C^1(D))=:\Hig_<^{\nu}(C,D).
$$
Since endomorphisms of $(E_1,\textbf{p})$ are scalars and $\Omega_C^1\simeq \mathcal O_C$, this vector space is two dimensional. Then, the fiber of the forgetful map $Bun : {\rm Con}_<^{\nu}(C,D) \longrightarrow X_<$ 
over $(E_1,\textbf{p})$ identifies with the two dimensional affine space 
$$
Bun^{-1}(E_1,\textbf{p})\simeq \nabla^0+{\rm H}^0(C,{\rm End}(E_1,\textbf{p})\otimes\Omega_C^1(D)),
$$
where $(E_1,\nabla^0,\textbf{p})$ is any parabolic connection belonging to the fiber. 
We can compactify the affine $\mathbb C^2$-bundle $Bun : {\rm Con}_<^{\nu}(C,D) \longrightarrow X_<$ by compactifying  the  fiber 
$$
\overline{Bun^{-1}(E_1,\textbf{p})} = \mathbb P \left(\mathbb C\cdot \nabla^0 \oplus {\rm H}^0(C,{\rm End}(E_1,\textbf{p})\otimes 	\Omega_C^1(D))\right).
$$  
Varying 	$(E_1,\textbf{p})\in X_<$ and choosing a local section $\nabla^0$ over local open sets of $X_<$ we construct a $\mathbb P^2$-bundle 
$$
Bun : \overline{{\rm Con}_<^{\nu}(C,D)} \longrightarrow X_<\simeq \mathbb P^1\times \mathbb P^1
$$
where 
$$
\mathbb P\Hig_<^{\nu}(C,D):=\overline{{\rm Con}_<^{\nu}(C,D)}\backslash{\rm Con}_<^{\nu}(C,D)
$$
is the moduli space of projective Higgs fields. 

\subsection{From logarithmic connections to fuchsian systems with five poles}

In order to study the $\mathbb P^2$-bundle 
$$
Bun : \overline{{\rm Con}_<^{\nu}(C,D)} \longrightarrow X_<
$$
we perform three elementary transformations on $E_{1}$ to arrive on the trivial vector bundle. Then, we use global coordinates to obtain an explicit universal family for our moduli space.  The idea is the following. For  any triple $$(E, \nabla, \textbf{p})\in \overline{{\rm Con}_<^{\nu}(C,D)}$$ 
the underlying
rank $2$ bundle is always  $E=E_1$. It is well known that $E_1$ can be obtained from the trivial vector bundle $\mathcal O_C\oplus \mathcal O_C$ by three elementary transformations on distinct basis and distinct fibers (see for instance \cite[Theorem 4.8]{Ma1}). Therefore, to give a logarithmic connection on $E_{1}$ is equivalent to give a fuchsian system on the trivial bundle having three apparent singular points.  This will be explained in the next few lines. 

We say that $t\in C$ is an {\bf apparent singular} point for $\nabla$ if  
\begin{enumerate}
\item the residual part ${\rm Res}_t \nabla$ has $\{\frac{1}{2}, -\frac{1}{2}\}$ as eigenvalues; and 
\item {\color{black}{the $\frac{1}{2}$-eigenspace of ${\rm Res}_t \nabla$ is also invariant by the constant part of the connection matrix. }} 
\end{enumerate}
{\color{black}{These conditions does not depend on the choice of local trivialization for $E$; condition (2) is equivalent 
to say that the local monodromy is semi-simple, i.e. $\pm I$.}}

\begin{remark}\label{rem:App3Cond}We shall note that if $\nabla$ is an $\frak{sl}_2$ connection, requiring that $t$ is an apparent singularity with a given direction $p$ imposes three linear conditions on the coefficients of $\nabla$. In fact, if we denote by $A_{-1}, A_0 \in GL(2,\mathbb C)$ the residual and the constant part of the connection matrix, respectively, then $(i)$ and $(ii)$ above means
\begin{displaymath}
\left\{ \begin{array}{ll}
\left(A_{-1} - \frac{1}{2} I\right)\cdot p  = 0\\
(A_0\cdot p)\wedge p = 0
\end{array} \right.
\end{displaymath}
in which gives us three linear conditions.
\end{remark}

We note that local monodromy does not change when we perform an elementary transformation over an apparent singular point, but the residual matrix becomes a multiple of the identity. In fact, let us assume that $x=t$ is an apparent singular point where the kernel of the residual part is  \begin{eqnarray*}
\textbf{p}=\left(
\begin{array}{ccc} 
1   \\
0 \\
\end{array}
\right).
\end{eqnarray*}
Then around $x=t$ the connection matrix writes 
\begin{eqnarray*}
\nabla= d + \left(
\begin{array}{ccc} 
\frac{1}{2} & \beta_{-1}  \\
0 & -\frac{1}{2} \\
\end{array}
\right)\cdot \frac{dx}{(x-t)} +
\left(
\begin{array}{ccc} 
\alpha_0 & \beta_0  \\
0 & -\alpha_0 \\
\end{array}
\right)\cdot dx + o(x-t)
\end{eqnarray*}
Applying an elementary transformation 
\begin{eqnarray*}
elm^+ (Y)=\left(
\begin{array}{ccc} 
x & 0  \\
0 & 1 \\
\end{array}
\right)\cdot Y
\end{eqnarray*}
we see that the residual matrix becomes 
\begin{eqnarray*}
 \left(
\begin{array}{ccc} 
-\frac{1}{2} & 0  \\
0 & -\frac{1}{2} \\
\end{array}
\right). 
\end{eqnarray*}

We now fix the data in order to  define the moduli space of fuchsian systems with five poles on $C$ and three apparent singular points.  Consider the $2$-torsion points $w_i=(i,0) \in C$, for $i=0,1,\lambda$, defined in (\ref{torsionpoints}).
The divisor defined by them is linearly equivalent to $ 3w_{\infty}$
$$
w_{0}+w_{1}+w_{\lambda} \sim 3w_{\infty}.
$$ 
Let us consider the reduced divisor 
$$
D'=w_0+w_1+w_{\lambda}+t_1+t_2
$$
and fix a local exponent 

\begin{eqnarray}\label{theta}
\theta=\left (\pm \frac{1}{2}, \pm \frac{1}{2}, \pm \frac{1}{2}, \pm \frac{\nu_{1}}{2}, \pm \frac{\nu_{2}}{2}\right )
\end{eqnarray}

We denote by ${\rm Syst}^{\theta}(C,D')$  the moduli space of $\mathfrak {sl}_2$-Fuchsian systems 
(i.e. logarithmic $\mathfrak {sl}_2$-connections on the trivial bundle) having $D'$ as divisor of poles, 
$\theta$ as local exponents and such that 
\begin{itemize}
\item the three singular points $w_0, w_1$ and $w_{\lambda}$ are apparent singular points;
\item the corresponding $\frac{1}{2}$-eigenspaces $q_i$ are pairwise distinct.
\end{itemize}
In other words, up to isomorphism of the trivial vector bundle, one can assume $p_i=(w_i, (i,0))$, for $i=0,1,\lambda$. 

\begin{prop}\label{systems}
There is an isomorphism of moduli spaces
\begin{eqnarray}\label{isomorphism}
{\rm Syst}^{\theta}(C,D')\stackrel{\sim}{\longrightarrow} {\rm Con_{<}^{\nu}}(C,D)
\end{eqnarray}
\end{prop}

\begin{center}
\begin{figure}[h]
\centering
\includegraphics[height=2in]{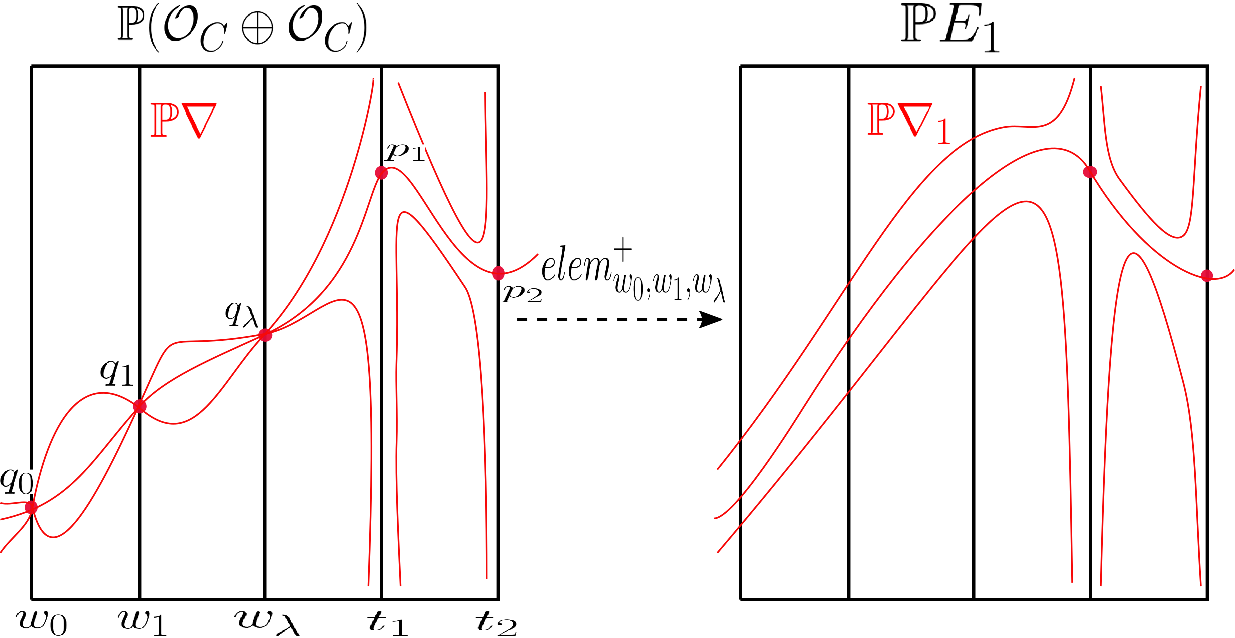}
\caption{Elementary Transformation}
\label{elementary}
\end{figure}
\end{center}

\proof
Let us describe the isomorphism (\ref{isomorphism}); let $(E,\nabla, \textbf{p})$ be a point in ${\rm Syst}^{\theta}(C,D')$.
Consider the composition of three positive elementary transformations 
\begin{eqnarray}\label{elm}
el= elem^{+}_{w_0,w_1,w_{\lambda}}: \mathcal O_C\oplus \mathcal O_C \dashrightarrow \tilde{E}
\end{eqnarray}
on distinct basis points $\{w_0,w_1,w_{\lambda}\}$ and distinct fibers $\{q_{0}, q_{1}, q_{\lambda}\}$ (see Figure \ref{elementary}).
After this birational modification of the trivial bundle, the corresponding indecomposable rank 2 vector bundle $\tilde{E}$ has determinant 
$$
{\rm det}\tilde{E}=\mathcal O_{C}(w_{0}+w_{1}+w_{\lambda}).
$$
Since $w_{0}+w_{1}+w_{\lambda}$ is linearly equivalent to $3w_{\infty}$ then 
$$
{\rm det}\tilde{E}=\mathcal O_{C}(3w_{\infty}).
$$
Applying such composition of elementary transformations (\ref{elm}) to $(\mathcal O_C\oplus \mathcal O_C, \nabla)$ one obtains a new pair $(\tilde{E}, \tilde{\nabla})=el_*(\mathcal O_C\oplus \mathcal O_C, \nabla)$.  The eigenvalues $\{\frac{1}{2}, -\frac{1}{2}\}$ of $\nabla$ over $w_i$,  are changed by
\begin{eqnarray*}
\tilde{\nu}_i^{+}&=&\frac{1}{2}-1=-\frac{1}{2} \\
\tilde{\nu}_i^{-}&=&-\frac{1}{2}
\end{eqnarray*}
and other eigenvalues $\{  \frac{\nu_{i}}{2}, -\frac{\nu_{i}}{2}\}$ over $t_i$, for $i=1,2$, are left unchanged.

Then the birational map $el$ yields a new connection  $\tilde{\nabla}$ on $\tilde{E}$
$$
\tilde{\nabla}: \tilde{E} \longrightarrow  \tilde{E} \otimes \Omega_C^1(w_0+w_1+w_{\lambda}+t_1+t_2)
$$
with local exponents 
$$
\tilde{\nu}=\left (\tilde{\nu}_1^{\pm}=-\frac{1}{2},\tilde{\nu}_2^{\pm}=-\frac{1}{2},\tilde{\nu}_3^{\pm}=-\frac{1}{2}, \pm \frac{\nu_{1}}{2},\pm \frac{\nu_{2}}{2}\right ).
$$ 
Since singularities over $w_{i}$ are apparent,  we promptly deduce that $\tilde{\nabla}$ has local monodromy $-Id$ over them. That is, these singular points  are projectively apparent.  
Its trace connection ${\rm tr}(\tilde{\nabla})$ is the unique rank one connection $d-\frac{dy}{y}$ on 
$$
\mathcal O_C(w_0+w_1+w_\lambda)=\mathcal O_C(3w_{\infty}).
$$
with trivial monodromy, poles on $w_0,w_1,w_\lambda$ and exponents $-1$.
In order to restore a connection on $E_1$ without singularities over $w_i$, for $i=0,1,\lambda$, 
and with trace $\zeta$, we have to twist $(\tilde{E},\tilde{\nabla})$ by a suitable logarithmic rank one connection $\xi$
on $\mathcal O_C(-w_{\infty})$, namely a square root $\xi$ of
\begin{equation}\label{eq:theTwist}
(\mathcal O_C(-w_{\infty}),\xi)^{\otimes2}:=(\mathcal O_C(w_{\infty}),\zeta)\otimes \left(\mathcal O_C(3\omega_\infty),\zeta\otimes(d-\frac{dy}{y})\right)^{\otimes(-1)}
\end{equation}
$$=\left(\mathcal O_C(-2\omega_\infty),\zeta\otimes(d+\frac{dy}{y})\right)
$$
where $\zeta$ is the fixed trace connection in the moduli space ${\rm Con_{<}^{\nu}}(C,D)$.
The resulting logarithmic connection $\nabla_{1}=\tilde{\nabla}\otimes \xi$ on $E_1=\tilde{E}\otimes \mathcal O_C(-w_{\infty})$ is nonsingular over $w_i$, for $i=0,1,\lambda$, and has eingenvalues $\{\frac{\nu_{1}}{2}-\frac{1}{2}, -\frac{\nu_{1}}{2}-\frac{1}{2}\}$ over $t_1$ and $\{\frac{\nu_{2}}{2}, -\frac{\nu_{2}}{2}\}$ over $t_2$. The parabolic structure ${\textbf{p}_{1}}$ on $E_{1}$ over $D=t_{1}+t_{2}$ is defined by the image by $el$ of the parabolic structure on $\mathcal O_C\oplus \mathcal O_C$ forgetting the directions over $w_{i}$, $i=0, 1, \lambda$. 

Then for each element $(E,\nabla, \textbf{p})$ in ${\rm Syst}^{\theta}(C,D')$, we have associated an element $(E_{1},\nabla_{1}, \textbf{p}_{1})$ in ${\rm Con}^{\nu}(C,D)$. This process can be reversed, and one gets an isomorphism
\begin{eqnarray*}
{\rm Syst}^{\theta}(C,D') &\stackrel{\sim}{\longrightarrow}& {\rm Con_{<}^{\nu}}(C,D)\\
			(E,\nabla, \textbf{p}) &\mapsto& (E_{1},\nabla_{1}, \textbf{p}_{1})
\end{eqnarray*}
In fact, the inverse of transformation $E_1\dashrightarrow\mathcal O_C\oplus \mathcal O_C$ obtained
by reversing (\ref{elm}) and (\ref{eq:theTwist}) can be described as follows.
Let $L_i\subset E_1$ denotes the unique embedding of $\mathcal O_C(w_i-w_\infty)$ for $i=0,1,\lambda$.
Then, $L_i$ is transformed in $\mathcal O_C\oplus \mathcal O_C$ as the constant (trivial) subbundle generated 
by $(i,0)$ for $i=0,1,\lambda$. This characterizes the inverse transformation.
\endproof

\begin{remark}\label{trivial line bundle}
As shown in Figure \ref{sections of E_1}, the subbundles $L_0,L_1,L_\lambda\subset E_1$ previously defined
correspond, after projectivization, to sections of $\mathbb P E_1$ having $+1$ self-intersection. 
We similarly define $L_\infty\subset E_1$ as the unique trivial subbundle $\mathcal O_C \hookrightarrow E_1$. 
Through the transformation $\mathcal O_C\oplus \mathcal O_C\dashrightarrow E_1$ obtained by composing 
birational modification $el$  (\ref{elm}) and twist by (\ref{eq:theTwist}), one easily check that the subbundle $L_\infty\simeq\mathcal O_C$ of $E_1$corresponds
to the subbundle $\mathcal O_C(-w_0-w_1-w_\lambda+w_{\infty})\hookrightarrow\mathcal O_C\oplus \mathcal O_C$ generated by the rational section
$$(x,y)\mapsto \begin{pmatrix}1\\ x\end{pmatrix}.$$
This subbundle corresponds, after projectivization, to the unique section 
$$s: C\longrightarrow C\times \mathbb P^1$$ 
having $+4$ self-intersection, passing through the points 
$$q_i=((i,0), (1:i)),\ \ \ i=0,1,\lambda$$ 
and which is invariant under the hyperelliptic involution.

\begin{center}
\begin{figure}[h]
\centering
\includegraphics[height=1.6in]{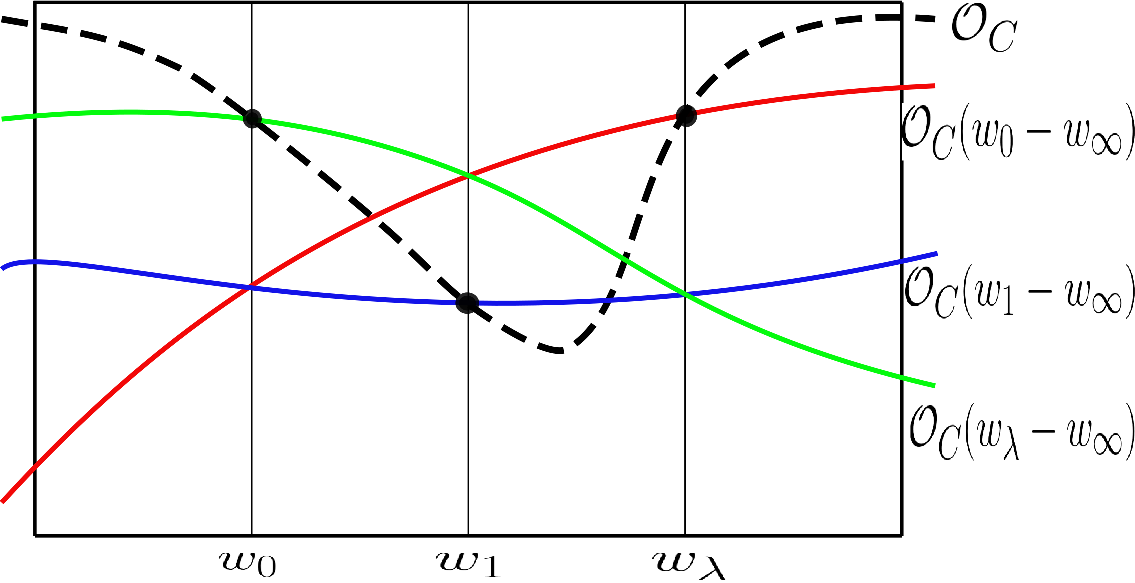}
\caption{Sections of $\mathbb P E_1$}
\label{sections of E_1}
\end{figure}
\end{center}
\end{remark}

\subsection{Universal family of fuchsian systems}\label{sec:universal}
As we have seen in Proposition \ref{systems} the open subset  ${\rm Con}_<^{\nu}(C,D)$ of the moduli space  ${\rm Con}^{\nu}(X,D)$  is isomorphic to the moduli space of fuchian systems ${\rm Syst}^{\theta}(C,D')$ with three apparent singularities over $w_{i}$ and local exponent $\theta$ as in (\ref{theta}). In this section, we exhibit an explicit universal family for ${\rm Syst}^{\theta}(C,D')$, when restricted to two suitable open sets of parabolic bundles.  As a consequence, we determine the  $\mathbb P^2$-bundle 
$$
Bun : \overline{{\rm Con}_<^{\nu}(C,D)} \longrightarrow X_<\simeq \mathbb P^1\times \mathbb P^1.
$$

Firstly we would like to write  $X_<\simeq \mathbb P^1\times \mathbb P^1$, minus two points,  as a union of two copies $U_0\cup U_{\infty}$ of $ \mathbb C^2$ (see below).  The $\mathbb P^2$-bundle above 
will be trivial when restricted to each one of them.

Let $U_0$ and $U_{\infty}$ be defined as follows
$$
U_0 = \left\{ \z=((z_1:1),(z_2:1)) \; ; \;\; z_1,z_2\in\mathbb C \right\}
$$
$$
U_{\infty} = \left\{ ((1:Z_1),(1:Z_2)) \; ; \;\; Z_1,Z_2\in\mathbb C \right\}.
$$
Note that $U_0\cup U_{\infty}$ cover $\mathbb P^1\times \mathbb P^1$ minus two points: $$((1:0),(0:1)) \; \text{and}\;((0:1),(1:0)).$$

Each point $\z \in U_{0}\subset \mathbb P^{1} \times \mathbb P^{1}$ corresponds to a parabolic vector bundle $(E_{1}, \z)$ in $X_{<}$. By using the compositions of elementary transformations (\ref{elm}), one can identify each $(E_{1}, \z)$ in $X_{<}$ with a parabolic bundle $(\mathcal O_C\oplus \mathcal O_C, \textbf{p}_{\z})$ where the parabolic structure is
$$
\textbf{p}_{\z}:= \{q_{0}, q_{1}, q_{\lambda}, z_1,z_2\}.
$$ 
Thanks to Proposition  \ref{systems}, we identify the fiber $Bun^{-1}(E_{1}, \z)$ with the space of fuchian systems in ${\rm Syst}^{\theta}(C,D')$ over the same fixed parabolic bundle $$(\mathcal O_C\oplus \mathcal O_C, {\textbf{p}_{\z}})$$
for particular $\theta$ as in (\ref{theta}).

Any fuchsian system $\nabla\in{\rm Syst}^{\theta}(C,D')$ writes 
\begin{eqnarray*}
\nabla= d + \left(
\begin{array}{ccc} 
\alpha & \beta  \\
\gamma & -\alpha \\
\end{array}
\right)\cdot \frac{dx}{y} 
\end{eqnarray*}
where $\omega = dx/y$ is a global regular $1$-form on $C$ and $\alpha, \beta, \gamma$ are rational functions on $C$  with at most simple poles on 
$$
D'=w_{0}+w_{1}+w_{\lambda}+t_{1}+t_{2}.
$$
The vector space that parametrizes the space of such rational functions is five dimension: 
$$
\alpha, \beta, \gamma \in {\rm H}^{0}(C,\mathcal O_{C}(w_{0}+w_{1}+w_{\lambda}+t_{1}+t_{2})).
$$
Then, we have $15$ parameters for $\alpha, \beta$ and $\gamma$. Each eigenvector $q_{i}$ over the apparent singular point $w_{i}$ imposes three linear conditions on the coefficients of $\nabla$, for each $i=0, 1, \lambda$ (see Remark \ref{rem:App3Cond}). Similarly, having  $z_1$ (or $z_2$) as eigenvector with corresponding $\nu_{1}^{+}$ (or $\nu_{2}^{+}$) as eigenvalue imposes  two linear conditions.

In order to find a basis for the moduli space of fuchsian systems $\nabla(\z)\in{\rm Syst}^{\theta}(C,D')$ with given parabolic structure $\textbf{p}_{\z}$, we have solved this linear system with $13$ equations and $15$ variables. 

Let us denote by $t_{1}=(t,\yt)$ and $t_{2}=(t,-\yt)$. For any $\z\in U_{0}$, a fuchsian system $\nabla(\z)$ like above writes
$$
\nabla(\z) = \nabla^{0}(\z) + c_{1}\cdot \Theta_{1}^{0}(\z)+ c_{2}\cdot\Theta_{2}^{0}(\z)
$$
for $c_{1}, c_{2} \in \mathbb C$, where $\nabla^{0}(\z)$ is a particular fuchsian system as well as  $\Theta_{1}^{0}(\z)$, $\Theta_{2}^{0}(\z)$ are particular Higgs fields given by Table \ref{basis}. 

In order to obtain also a universal family over the other open set $U_{\infty}$ of $X_{<}$,  for each $(Z_1,Z_2)=(\frac{1}{z_1},\frac{1}{z_2})\in U_{\infty}$ the respective fuchsian and Higgs fields are
\begin{displaymath}
\left\{ \begin{array}{ll}
\nabla^{\infty}(Z_1,Z_2) = \nabla^{0}(\frac{1}{Z_1},\frac{1}{Z_2}) +\frac{\nu_1}{2}\cdot Z_1 \cdot \Theta_1^{0}(\frac{1}{Z_1},\frac{1}{Z_2}) + \frac{\nu_2}{2}\cdot Z_2 \cdot \Theta_2^{0}(\frac{1}{Z_1},\frac{1}{Z_2}) \\
\Theta_1^{\infty}(Z_1,Z_2) = (Z_1)^2\cdot\Theta_1^{0}(\frac{1}{Z_1},\frac{1}{Z_2})\\
\Theta_2^{\infty}(Z_1,Z_2) = (Z_2)^2\cdot\Theta_2^{0}(\frac{1}{Z_1},\frac{1}{Z_2})
\end{array} \right.
\end{displaymath}
Here we shall explain why the parameters $\nu_1, \nu_2$ appear in the definition of $\nabla^{\infty}$. Note that $\nabla^0$ is a rational section of ${\rm Con_<}^{\nu}(C,D) \longrightarrow X_<$. In fact, $\nabla^{0}(\frac{1}{Z_1},\frac{1}{Z_2})$ has $\{Z_1Z_2=0\}$ as pole. This can be checked by using the explicit expression for $\nabla^0$ given in Table \ref{basis}. Then the coeffients $\frac{\nu_1}{2}, \frac{\nu_2}{2}$ appearing in the definition of $\nabla^{\infty}$ have been chosen to make $\nabla^{\infty}$ regular.

\begin{table}
\centering
\begin{tabular}[p]{|c|}
   \hline
$
\nabla^{0}= d + \left(
\begin{array}{ccc} 
\alpha & \beta  \\
\gamma & -\alpha \\
\end{array}
\right)\cdot \frac{dx}{y} 
$ \\
   \hline 
   $\alpha = \frac{1}{4}\left (\frac{((\lambda+(x-\lambda-1)t)x(\nu_1+\nu_2)-((\lambda+1)x^2+(-\lambda-(\lambda+1)t)x+t\lambda))}{(x-t)y}\right )$ \\
    $+ \frac{1}{4}\left (\frac{(\lambda+(t-\lambda-1)x)t(\nu_1-\nu_2)}{(x-t)\yt}-\frac{2x(\nu_1z_1+\nu_2z_2)}{y}+\frac{(\lambda-t+1)(\nu_1z_1-\nu_2z_2)}{\yt}\right )$ \\ \hline
   $\beta = \frac{1}{4}\left (\frac{2(\nu_1z_1+\nu_2z_2)-(\lambda+1-x)(\nu_1+\nu_2+1)+2x}{y}-\frac{2(t(\nu_2-\nu_1)+\nu_1z_1-\nu_2z_2)}{\yt} \right )$ \\ \hline
   $ \gamma = \frac{x}{4}\left (\frac{2(\lambda-x(1+\lambda-t))(\nu_1z_1+\nu_2z_2)}{y(x-t)}+\frac{\lambda(\nu_1+\nu_2-1)}{y}+\frac{2\yt(\nu_1z_1-\nu_2z_2)}{(x-t)t}\right )$ \\ \hline \hline
   $ \Theta_{1}^0=  \left(
\begin{array}{ccc} 
\alpha_1 & \beta_1  \\
\gamma_1 & -\alpha_1 \\
\end{array}
\right)\cdot \frac{dx}{y}$  \\ \hline
    $\alpha_1 = \frac{-1}{2(x-t)}\left (\frac{2z_1x\left ((t-z_1)x+t(z_1-1)+\lambda(1-t)\right )}{y}+\frac{\left (t\lambda+(2t-z_1)(t-\lambda-1)z_1\right )x}{\yt}\right )$ \\
	$+\frac{1}{2(x-t)}\left (\frac{t((t-2z_1)\lambda-(t-\lambda-1)z_1^2)}{\yt}\right )$ \\ \hline
	$\beta_1 = \frac{x(t-z_1)  \left (x+z_1-\lambda-1\right )+t(z_1-\lambda)(z_1-1)}{y(x-t)}+\frac{(z_1-t)^2x+t((2t-z_1)z_1+\lambda-t(\lambda+1))}{\yt(x-t)}$\\ \hline 
	$\gamma_1 = \frac{z_1x}{(x-t)}\left ( \frac{x\left (z_1(1-t)+\lambda(z_1-1)\right )+\lambda(t-z_1)}{y}-\frac{\yt z_1}{t} \right )$ \\\hline\hline
	$\Theta_{2}^0=  \left(
\begin{array}{ccc} 
\alpha_2 & \beta_2  \\
\gamma_2 & -\alpha_2 \\
\end{array}
\right)\cdot \frac{dx}{y} $ \\ \hline
$\alpha_2 = \frac{-1}{2(x-t)}\left (\frac{2z_2x((t-z_2)x+t(z_2-1)+\lambda(1-t))}{y}-\frac{((t\lambda+(2t-z_2)(t-\lambda-1)z_2)x}{\yt}\right )$\\
	$+\frac{-1}{2(x-t)}\left (\frac{t((t-2z_2)\lambda-(t-\lambda-1)z_2^2)}{\yt}\right )$\\ \hline
$\beta_2 = \frac{x(t-z_2)(x+z_2-\lambda-1)+t(z_2-\lambda)(z_2-1)}{y(x-t)}-\frac{(z_2-t)^2x+t\left ((2t-z_2)z_2+\lambda-t(\lambda+1)\right )}{\yt(x-t)}$\\ \hline
	$\gamma_2 = \frac{z_2x}{(x-t)}\left (\frac{x\left (z_2(1-t)+\lambda(z_2-1)\right )+\lambda(t-z_2)}{y}+\frac{\yt z_2}{t}\right )$\\ \hline
   \end{tabular}
\caption{Basis for the universal family over $U_{0}$}
\label{basis}
\end{table}

We summarize the above discussion in the next proposition.  The intrinsic meaning of these basis will be given in the next section. 

\begin{prop}\label{trivialization}
For each $j\in\{0,\infty\}$, the $\mathbb P^{2}$-bundle $Bun : \overline{{\rm Con}_<^{\nu}(C,D)} \longrightarrow X_<$ is trivial when restricted to $U_{j}$:
\begin{eqnarray*}
U_j\times \mathbb P^2 &\stackrel{\simeq}{\longrightarrow}& \overline{{\rm Con}_<^{\nu}(C,D)}|_{U_j} \\
	(u_{j} , (C_0:C_1:C_2)) &\mapsto& \mathbb P [C_0\cdot\nabla^j(u_{j})+C_1\cdot\Theta_1^j(u_{j})+C_2\cdot\Theta_2^j(u_{j})]
\end{eqnarray*} 
and the open set ${\rm Con}_<^{\nu}(C,D)|_{U_j}$ is given by $(C_0:C_1:C_2)=(1:c_1:c_2)$.
\end{prop}

\subsection{An open set of the moduli space}
Without loss of generality, we can assume that $\nu$ is as in (\ref{nu})
\begin{eqnarray}\label{nu fixed}
(\nu_1^{\pm},\nu_2^{\pm})=\left (\pm \frac{\nu_{1}}{2} - \frac{1}{2}, \pm \frac{\nu_{2}}{2}\right )
\end{eqnarray}
(see Section \ref{moduli space}). The condition $\nu_k^{+}=\nu_k^{-}$ is equivalent to $\nu_k=0$ for each $k\in \{1,2\}$.

\subsubsection{Description of $\overline{{\rm Con}_<^{\nu}(C,D)}$} 
We now apply  results of Section \ref{sec:universal} to study the $\mathbb P^{2}$-bundle 
$$
Bun : \overline{{\rm Con}_<^{\nu}(C,D)} \longrightarrow X_<\simeq \mathbb P^1_{z_1}\times \mathbb P^1_{z_2}
$$
as well as the projectivized moduli space of Higgs fields
$$
\mathbb P\Hig_<^{\nu}(C,D):=\overline{{\rm Con}_<^{\nu}(C,D)}\backslash{\rm Con}_<^{\nu}(C,D).
$$

\begin{thm}\label{main1}
With  identification $X_<\simeq \mathbb P_{z_1}^1\times \mathbb P_{z_2}^1$, the following assertions hold true:
\begin{enumerate}
\item  The canonical isomorphism 
$
\Hig_<^{\nu}(C,D) \stackrel{\sim}{\longrightarrow} T^*X_<
$ 
is given over charts $U_0$ and $U_\infty$  by the identifications
\begin{equation}\label{canisomHiggserredual}
\left\{\begin{matrix} \Theta_1^0 &\mapsto& dz_1 \\ \Theta_2^0 &\mapsto& dz_2\end{matrix}\right.
\ \ \ \text{and}\ \ \ 
\left\{\begin{matrix} \Theta_1^\infty &\mapsto& dZ_1 \\ \Theta_2^\infty &\mapsto& dZ_2\end{matrix}\right.
\end{equation}
respectively.
The Liouville form is $c_1dz_1+c_2dz_2$ and its differential 
$$
\omega=dc_1\wedge dz_1+dc_2\wedge dz_2
$$
induces the canonical symplectic structure on $\Hig_<^{\nu}(C,D)\vert_{U_0}$.
\item We have $\overline{{\rm Con}_<^{\nu}(C,D)}=\mathbb P(\mathcal E_{1}^{\nu})$,  where $\mathcal E_{1}^{\nu}$ is the extension of $\mathcal O_{X_{<}}$ by $T^*X_{<}$ 
$$
0\longrightarrow T^*X_{<}  {\longrightarrow}  \mathcal E_{1}^{\nu} {\longrightarrow}  \mathcal O_{X_{<}} \longrightarrow 0
$$
determined by 
$$
\left(\frac{\nu_{1}}{2}, \frac{\nu_{2}}{2}\right)\in {\rm H}^{1}(X_{<}, T^*X_{<})\simeq {\rm H}^{1}( \mathbb P^1_{z_1}, T^*\mathbb P^1_{z_1})\oplus {\rm H}^{1}(\mathbb P^1_{z_2}, T^*\mathbb P^1_{z_2} )\simeq \mathbb C^{2}.
$$
\end{enumerate}
\end{thm}

\proof
The isomorphism of (1) comes from Serre duality (see \cite[Section 6]{ViktoriaFrank}):
$$
\Hig_<^{\nu}(C,D)={\rm H}^0(C,{\rm End}(E_1,\textbf{p})\otimes \Omega_C^1(D))\ \simeq\ 
{\rm H}^1(C,{\rm End}(E_1,\textbf{p}))\simeq T^*X_{<}
$$
and is well-known. By the same argument as in \cite[Prop 6.1]{ViktoriaFrank}, we can check the identification (\ref{canisomHiggserredual}) therefore proving (1).

It follows by Proposition \ref{trivialization} that $\{\nabla^{j}, \Theta_1^j, \Theta_2^j\}$, for each $j\in\{0,\infty\}$, are sections trivializing $\overline{{\rm Con}_<^{\nu}(C,D)}|_{U_{j}}$ 
\begin{eqnarray*}
U_j\times \mathbb P^2 &\stackrel{\simeq}{\longrightarrow}& \overline{{\rm Con}_<^{\nu}(C,D)}|_{U_j} \\
	(u_{j} , (C_0:C_1:C_2)) &\mapsto& [C_0\cdot\nabla^j(u_{j})+C_1\cdot\Theta_1^j(u_{j})+C_2\cdot\Theta_2^j(u_{j})].
\end{eqnarray*} 
By construction, in the intersection $U_{0}\cap U_{\infty}$ they satisfy 
$$
\left\{ \begin{array}{ll}
\nabla^{\infty} = \nabla^{0} +\frac{\nu_1}{2}\cdot Z_1 \cdot \Theta_1^{0} + \frac{\nu_2}{2}\cdot Z_2 \cdot \Theta_2^{0} \\
\Theta_1^{\infty}= (Z_1)^2\cdot\Theta_1^{0}\\
\Theta_2^{\infty} = (Z_2)^2\cdot\Theta_2^{0}
\end{array} \right.
$$
Therefore, transition chart from $U_{\infty}\times \mathbb C^{3}$ to $U_{0}\times \mathbb C^{3}$ give us the following  cocycle 
\begin{eqnarray*}
\left(
\begin{array}{ccc} 
1&0&0\\
\frac{\nu_{1}}{2}\cdot Z_1&Z_1^{2} & 0   \\
\frac{\nu_{2}}{2}\cdot Z_2&0 & Z_2^{2}  
\end{array}
\right).
\end{eqnarray*}
Since $U_{0}\cup U_{\infty}$ covers the base $X_{<}\simeq \mathbb P^{1}\times \mathbb P^{1}$ minus two points, this cocycle completely determines the extension. This proves that  $\mathcal E_{1}^{\nu}$ corresponds to 
$$
\left(\frac{\nu_{1}}{2}, \frac{\nu_{2}}{2}\right)\in {\rm H}^{1}(\mathbb P^1_{z_1}\times \mathbb P^1_{z_2}, T^*(\mathbb P^1_{z_1}\times \mathbb P^1_{z_2}) )\simeq {\rm H}^{1}( \mathbb P^1_{z_1}, T^*\mathbb P^1_{z_1})\oplus {\rm H}^{1}(\mathbb P^1_{z_2}, T^*\mathbb P^1_{z_2} )\simeq \mathbb C^{2}.
$$
This concludes item $(2)$ of the statement.
\endproof

We deduce the following corollary.

\begin{cor}\label{trivial extension}
The $\mathbb P^{2}$-bundle  $Bun : \overline{{\rm Con}_<^{\nu}(C,D)} \longrightarrow X_<$ 
is the trivial extension
$$
\overline{{\rm Con}_<^{\nu}(C,D)} \simeq \mathbb P (T^*X_{<}\oplus \mathcal O_{X_{<}})
$$
if and only if $\nu_{1}=\nu_{2}=0$. 
\end{cor}

In fact, when $\nu_{1}\cdot \nu_{2}\neq 0$, then one easily see that all these $\mathbb C^{2}$-bundle are isomorphic between them, they correspond to the extension defined by the cocycle 
\begin{eqnarray*}
\left(
\begin{array}{ccc} 
1&0&0\\
z_1&z_1^{2} & 0   \\
z_2&0 & z_2^{2}  
\end{array}
\right).
\end{eqnarray*} 
This will be done in a geometric way in the next section.

\subsubsection{The moduli space ${\rm Con}_<^{\nu}(C,D)$ for $\nu_{1}\cdot \nu_{2}\neq 0$}\label{differente de zero}

When $\nu_{1}\cdot \nu_{2}\neq 0$ the kernel of the residual part of $\nabla\in {\rm Con}_<^{\nu}(C,D)$ over each $t_{i}$ has two distinct eigendirections 
$$
\ker({\rm Res_{t_{i}}}\nabla-\nu_i^{\pm}I) =  
{p_{i}^\pm(\nabla)}\in \mathbb P^{1}
$$
where $p_{i}^{\epsilon}$ corresponds to the eigenvalue $\nu_i^{\epsilon}$ for $\epsilon\in\{+,-\}$.

Let us denote by $\Delta\subset \mathbb P^1\times \mathbb P^1$  the diagonal and $S:=(\mathbb P^1\times \mathbb P^1)\backslash\Delta$ its complement.  Then one gets a mapping which associates, to each $\nabla \in {\rm Con}_<^{\nu}(C,D)$, the eigenvectors of its residual part 
\begin{eqnarray*}
\Par: {\rm Con}_<^{\nu}(C,D) &\longrightarrow& S^2\\
                     \nabla            &\mapsto&   \left(({p}_{1}^{+}(\nabla), {p}_{1}^{-}(\nabla)), ({p}_{2}^{+}(\nabla), {p}_{2}^{-}(\nabla))\right)              
\end{eqnarray*}
In the next theorem, we will prove that this map is an isomorphism. For each $\epsilon\in \{+,-\}$, we shall denote by $\tau^{\epsilon}$ the  projection corresponding to the directions $\{p_{1}^{\epsilon}, p_{2}^{\epsilon}\}$:

\begin{eqnarray*}
\tau^{\epsilon}:S^{2} &\longrightarrow& \mathbb P^1\times \mathbb P^1.\\
		((p_{1}^{+}, p_{1}^{-}), (p_{2}^{+}, p_{2}^{-})) &\mapsto& (p_{1}^{\epsilon}, p_{2}^{\epsilon})
\end{eqnarray*}

Similarly, for each $\epsilon\in\{+,-\}$, there is a projection $Bun^{\epsilon}$ corresponding to the respective parabolic structure associated to eigenvalues $\{\nu_{1}^{\epsilon}, \nu_{2}^{\epsilon}\}$.
\begin{eqnarray*}
Bun^{\epsilon}: {\rm Con}_<^{\nu}(C,D) &\longrightarrow& X_{<}.\\
		(E_{1}, \nabla, \textbf{p}) &\mapsto& (E_{1}, \textbf{p}^{\epsilon}(\nabla))
\end{eqnarray*}
Note that  ${\rm Con}_<^{\nu}(C,D)$ has been defined by assuming that if $(E_{1}, \nabla, \textbf{p})$ is a point inside it then $\textbf{p}$ corresponds to the eingenvalues $\{\nu_{1}^{+}, \nu_{2}^{+}\}$. This means that the above projection $Bun^{+}$ coincides with $Bun$ of previous sections. 

\begin{thm}\label{isomorphism affine}
If $\nu_{1}\cdot \nu_{2}\neq 0$ then $\Par: {\rm Con}_<^{\nu}(C,D) \longrightarrow S^2$ is an isomorphism between $\mathbb C^2$-affine bundles over $\mathbb P^1\times \mathbb P^1$
\[
\xymatrix { 
{\rm Con}_<^{\nu}(C,D)  \ar@{->}[dr]_{Bun^{\epsilon}} \ar@{->}[rr]^{\Par}_{\simeq}  &    & S^2 \ar@{->}[dl]^{{\tau}^\epsilon} \\
      &          \mathbb P^1\times \mathbb P^1 & &
}
\]
for each $\epsilon\in\{+, -\}$.
\end{thm}

\proof
In order to prove the theorem we shall use the isomorphism
 $${\rm Syst}^{\theta}(C,D')\simeq {\rm Con_{<}^{\nu}}(C,D)$$
given by  Proposition \ref{systems}. We need to prove that the four eigenvectors $\{{\textbf{p}_1^{\pm}}(\nabla), {\textbf{p}_2^{\pm}}(\nabla)\}$ corresponding to the eigenvalues $\left\{\pm \frac{\nu_1}{2}, \pm \frac{\nu_2}{2}\right\}$ determine the connection.   

Since  $\nu_{1}\cdot \nu_{2}\neq 0$, one has two distinct eigenvectors over each singular point $t_1$ and $t_2$. Then we can assume, without loss of generality,  that ${\textbf{p}_1^{+}}(\nabla)=(1:z)$ and ${\textbf{p}_2^{+}}(\nabla)=(1:w)$. If ${\textbf{p}_i^{+}}(\nabla)=(0:1)$ for some $i=1,2$, the same argument works with $-\nu_i$ in place of $\nu_i$.

The fiber of $$Bun: {\rm Syst}^{\theta}(C,D') \longrightarrow X_<\simeq \mathbb P^1\times\mathbb P^1$$
over $\z\in U_0$ is the two dimensional affine space 
\begin{equation}\label{FormuleNablac}
\nabla_{\bc}(\z) = \nabla^{0}(\z) + c_{1}\cdot \Theta_{1}^{0}(\z)+ c_{2}\cdot\Theta_{2}^{0}(\z)
\end{equation}
varying $\bc=(c_1,c_2)\in\mathbb C^2$, where $\nabla_0$, $\Theta_{1}^{0}$ and $\Theta_{2}^{0}$ were given in Table \ref{basis}. Using the explicit expressions for them, we can compute the two eigenvectors of the kernel of 
$\nabla_{\bc}(\z)$ at each fiber over $t_1$ and $t_2$.  They are 
\begin{equation}\label{FormuleParc}
\Par\ :\ \nabla_{\bc}\ \mapsto\ \left\{\begin{matrix}
p_1^+=(1: z_1)&p_1^-=(c_1: c_1z_1-\frac{\nu_1}{2})\\
p_2^+=(1: z_2)& p_2^-=(c_2: c_2z_2-\frac{\nu_2}{2})
\end{matrix}\right.
\end{equation}
This proves  that  $\nabla_{\bc}(\z)$ is determined by its four residual eigendirections over $t_1$ and $t_2$.
\endproof

\begin{remark}\label{rk:characNabla0}The section $\nabla_0(\z)$ of $Bun^+$ is characterized by the property that $$\Par(\nabla_0)=\left\{\begin{matrix}
p_1^+=(1: z_1)&p_1^-=(0: 1)\\
p_2^+=(1: z_2)& p_2^-=(0: 1)
\end{matrix}\right.
$$
It thus coincides with a fiber of $Bun^-$.
\end{remark}

\begin{remark}
Now we would like to remark that there is an isomorphism between $S^2$ and $(Q\times Q^{*})\backslash \mathcal I$, where $Q\subset \mathbb P^3$ is the smooth quadric surface, $Q^*$ its dual variety and $\mathcal I$ is the incidence variety. As a consequence, we obtain the following isomorphism 
$$
{\rm Con}_<^{\nu}(C,D) \simeq (Q\times Q^{*})\backslash \mathcal I.
$$
In fact, let $\check{\mathbb P}^{3}$ be the set of hyperplanes in $\mathbb P^{3}$.   The incidence variety $\mathcal I$ is defined as
$$
\mathcal I=\{(q,T_rQ)\in Q\times Q^{*}\;;\;\; q\in T_rQ \}.
$$
We note that $Q$ and $Q^*$ are isomorphic. For instance, we can assume that $Q$ is defined by zero locus of $f:=x_{0}x_{3}-x_{1}x_{2}$. Its dual variety $Q^*$ can be identified with the image of the polar map $\phi_f:\mathbb P^3 \dashrightarrow \mathbb P^3$, which is defined by the derivatives of $f$. That is, $Q^*$ is just the image of $Q$ by a  linear map and it is determined by the same equation $f$ as $Q$. 

Via Segre embedding, we can identify $(Q\times Q^{*})\backslash \mathcal I$ with $(\mathbb P^{1}\times \mathbb P^{1})^{2}\backslash I$
where 
$$
I=\{ ((z_1,\zeta_1),(z_2,\zeta_2))\in (\mathbb P^{1}\times \mathbb P^{1})^{2} \;; \; z_1=z_2 \; \text{or}\; \zeta_1=\zeta_2\}.
$$
In addition, this last variety is isomorphic to $S^2$:
\begin{eqnarray*}
(\mathbb P^{1}\times \mathbb P^{1})^{2}\backslash I &\longrightarrow& S^2\\
   ((z_1,\zeta_1),(z_2,\zeta_2)) &\mapsto& ((z_1,z_2),(\zeta_1,\zeta_2)).
\end{eqnarray*}
\end{remark}

\subsection{Whole moduli space of connections}

In order to describe the whole moduli space, we have to introduce other open charts.

\subsubsection{Description of $\overline{{\rm Con}_>^{\nu}(C,D)}$} 
Recall that we have a stratification of the moduli space of simple parabolic bundles, see (\ref{strat}): 
\begin{eqnarray*}
{\rm Bun}_{w_{\infty}}(C,D) = X_<\sqcup X_>.
\end{eqnarray*}
Given $(E, \textbf{p})\in X_{>}$, we perform two elementary transformations $elm_{t_{1}, t_{2}}$ with center at $\textbf{p}$, followed by a twisting by $\mathcal O_{C}(-w_{\infty})$ in order to obtain a parabolic vector bundle $(\tilde{E}, \tilde{\textbf{p}})\in X_{<}$.  Such operation defines a mapping $\phi_{D}$ from $X_{>}$ to $X_{<}$ that sends the locus $\Gamma\subset X_{>}$ into the locus $\Gamma\subset X_{<}$. Identifying  $X_{>}$ and $X_{<}$ with $\mathbb P^{1}\times \mathbb P^{1}$ we can show that this map corresponds to the following one
\begin{eqnarray*}
 \phi_{D}: \mathbb P^{1}\times \mathbb P^{1} &\longrightarrow& \mathbb P^{1}\times \mathbb P^{1}\\
		(z_1,z_2)  &\mapsto& (z_2,z_1)
\end{eqnarray*}
(see \cite[Proposition 5.5]{Nestor} for details).

\begin{thm}\label{second open}
There is a fiber-preserving isomorphism $\Phi_{D}$:
\[
\xymatrix { 
\overline{{\rm Con}_>^{\nu}(C,D)} \ar@{->}[d] \ar@{->}[r]^{\Phi_{D}}_{\simeq}  &    \overline{{\rm Con}_<^{\lambda}(C,D)} \ar@{->}[d] \\
X_{>}  \ar@{->}[r]^{\phi_{D}}    &          X_{<}
}
\]
where $\lambda=(\lambda_{1}^{\pm}, \lambda_{2}^{\pm})$ with 
\begin{displaymath}
\left\{ \begin{array}{ll}
\lambda_{k}^{+} = \nu_{k}^{+}- 1/2\\
\lambda_{k}^{-} = \nu_{k}^{-}+1/2
\end{array} \right.
\end{displaymath}
for each $k\in\{1, 2\}$.
\end{thm}

\proof
The above isomorphism is given by performing a positive elementary transformation over $D=t_{1}+t_{2}$. In fact, given $(E, \nabla, \textbf{p}) \in {\rm Con}_>^{\nu}(C,D)$, let us  consider the composition of two positive elementary transformations 
\begin{eqnarray}\label{elm2}
el= elem^{+}_{t_{1}, t_{2}}: E \dashrightarrow E'
\end{eqnarray}
with center at $p_{1}^{+}$ and $p_{2}^{+}$, respectively. The birational map $el$ yields a new logarithmic connection $\nabla'$ on $E'$. We also shall fix a suitable
rank one logarithmic connection on $\mathcal O_C(-w_{\infty})$ with poles on $D$ satisfying the prescribed condition on residues:
$$
{\rm Res}_{t_k}(\zeta)=\frac{1}{2} 
$$
for $k\in \{1, 2\}$.
Twisting $(E', \nabla')$ by $(\mathcal O_{C}(-w_{\infty}), \zeta)$ one obtains a pair 
$$
(\tilde{E}, \tilde{\nabla}) = (E'\otimes \mathcal O_{C}(-w_{\infty}), \nabla'\otimes \zeta)
$$
where $\tilde{E}$ has  $\mathcal O_{C}(w_{\infty})$ as determinant line bundle and the resulting connection $\tilde{\nabla}$ has local exponents $\lambda$ as stated. 

Then we get a mapping $\Phi_{D}: {\rm Con}_>^{\nu}(C,D)  \longrightarrow {\rm Con}_<^{\lambda}(C,D)$ in which can be reversed by the respective negative elementary transformation followed by a twisting by $\mathcal O_{C}(w_{\infty})$. 
\endproof

\begin{cor}\label{maior}
The $\mathbb P^{2}$-bundle  $Bun : \overline{{\rm Con}_>^{\nu}(C,D)} \longrightarrow X_>$ is the trivial extension
$$
\overline{{\rm Con}_>^{\nu}(C,D)} = \mathbb P (T^*X_{>}\oplus \mathcal O_{X_{>}})
$$
if and only if $\nu_{k}^{+}-\nu_{k}^{-}=1$ for $k\in\{1, 2\}$.  If $\nu_{k}^{+} - \nu_{k}^{-}\neq 1$ for $k\in\{1, 2\}$, then ${\rm Con}_>^{\nu}(C,D)$ is isomorphic to $S^{2}$.
\end{cor}

\proof
First part of the statement is a consequence of Theorem \ref{second open} and Corollary \ref{trivial extension}. In the second part, we use Theorem \ref{isomorphism affine} to conclude the proof.  
\endproof

\subsubsection{Fiber over each indecomposable bundle}
In view of Corollary \ref{cor criterion}, any parabolic bundle arising in our moduli space ${\rm Con}^{\nu}(C,D)$ is indecomposable. There are two types of indecomposable parabolic bundles: simple and not simple. As we have seen in Section \ref{par}, the moduli space of simple parabolic bundle  having $\mathcal O_{C}(w_{\infty})$ as determinant line bundle is a union of two copies of $\mathbb P^{1} \times \mathbb P^{1}$
\begin{eqnarray*}
{\rm Bun}_{w_{\infty}}(C,D) = X_<\ \cup\  X_>
\end{eqnarray*}
identifying identical parabolic bundles outside a curve $\Gamma$ of type $(2,2)$. 
On the other hand, indecomposable not simple parabolic bundles were characterized in Proposition \ref{not simple}. There are eight different types: for each square root $L_{i,k}$ of $\mathcal O_{C}(w_{\infty}-t_{k})$, $k\in \{1,2\}$ and $i\in\{1,...,4\}$, there is a unique parabolic bundle up to automorphism 
\begin{eqnarray}\label{types}
\mathcal E_{i,k} =  (L_{i,k}\oplus L_{i,k}^{-1}(w_{\infty}), \textbf{p})
\end{eqnarray}
with parabolic directions as in Proposition \ref{not simple} (see Figure \ref{basepoint}).

Let $\mathcal C_{\mathcal E}\subset {\rm Con}^{\nu}(C,D)$ be the set of $\nu$-flat connections over a given parabolic bundle $\mathcal E =  (E, \textbf{p})$. The whole moduli space ${\rm Con}^{\nu}(C,D)$ is the union of $\mathcal C_{\mathcal E}$ when $\mathcal E$ runs over all indecomposable parabolic bundles. 

\begin{prop}\label{description}
Assume $\nu_{k}^{+}\neq \nu_{k}^{-}$ for each $k=1,2$. The moduli space $\mathcal C_{\mathcal E}$ of $\nu$-flat connections over a given indecomposable parabolic bundle $\mathcal E =  (E, \textbf{p})$ is a two dimensional affine space. 
\end{prop}

\proof
If $\mathcal E$ belongs to either $X_{<}$ or $X_{>}$ then $\mathcal E$ is simple. 
Since only automorphisms preserving parabolics are scalar, $\mathcal C_{\mathcal E}$ can be identified with the two dimensional affine space
$$
\mathcal C_{\mathcal E} \simeq {\rm H}^0(C,{\rm End}\mathcal E\otimes 	\Omega_C^1(D)).
$$
Now let $\mathcal E=\mathcal E_{i,k}$  as in (\ref{types}) be an indecomposable parabolic bundle but not simple. 
For sake of simplicity, one may assume that $\mathcal E=(L\oplus L^{-1}(w_{\infty}), \textbf{p})$ with $$L^{2} =\mathcal O_{C}(w_{\infty}-t_{1}).$$ Where the parabolic direction $\textbf{p} = (p_{1}, p_{2})$ satisfies
\begin{displaymath}
\left\{ \begin{array}{ll}
p_{1} \nsubseteq  L_{t_{1}} \,\, \text{and} \,\, p_{1} \nsubseteq L^{-1}(w_{\infty})_{t_{1}}\\
p_{2} \subset L^{-1}(w_{\infty})_{t_{2}}
\end{array} \right.
\end{displaymath}
like Figure \ref{basepoint}. By our hypothesis on $\nu$,   any $\nu$-flat connection $\nabla\in\mathcal C_{\mathcal E}$ has two distinct eigendirections over each $t_{k}$, for each $k=1,2$
$$
\Par({\rm Res_{t_{k}}}\nabla) =  {p_{k}^\pm(\nabla)}.
$$
Besides that, the ``positive'' directions coincide with $\textbf{p}$, that is, 
\begin{displaymath}
\left\{ \begin{array}{ll}
p_{1}^+(\nabla) = p_{1} \\
p_{2}^{+}(\nabla) = p_{2}. 
\end{array} \right.
\end{displaymath}
Also we note that since $p_{2} \subset L^{-1}(w_{\infty})_{t_{2}}$  then $p_{2}^-(\nabla) \nsubseteq L^{-1}(w_{\infty})_{t_{2}}$. We set 
$$q_{2}=p_{2}^-(\nabla)$$ 
to simplify the notation. Therefore, each element $\nabla\in\mathcal C_{\mathcal E}$ defines a connection on $L\oplus L^{-1}(w_{\infty})$ satisfying
 \begin{eqnarray}\label{-}
\left\{ \begin{array}{ll}
{\rm Res}_{t_1}(\nabla)(p_{1}) = \nu_1^+ \cdot p_{1} \\
{\rm Res}_{t_2}(\nabla)(q_{2}) = \nu_2^- \cdot q_{2}
\end{array} \right.
\end{eqnarray}
where 
\begin{displaymath}
\left\{ \begin{array}{ll}
p_{1} \nsubseteq  L_{t_{1}} \,\, \text{and} \,\, p_{1} \nsubseteq L^{-1}(w_{\infty})_{t_{1}}\\
q_{2} \nsubseteq L^{-1}(w_{\infty})_{t_{2}}.
\end{array} \right.
\end{displaymath}
A priori $q_{2}=p_{2}^-(\nabla)$ depends of $\nabla$. But since $p_{2}$ lies in the maximal subbundle $ L^{-1}(w_{\infty})$, the automorphism group of $E=L\oplus L^{-1}(w_{\infty})$ fixing parabolic $\textbf{p}=\{p_{1}, p_{2}\}$ is two dimensional
\begin{eqnarray*}
{\rm Aut}(E,\textbf{p}) = \left\{
\left(
\begin{array}{ccc} 
a & 0  \\
\gamma & a  \\
\end{array}
\right)
\;;\;\; a\in\mathbb C^{*},\; \gamma\in {\rm H}^0(C,\mathcal O_{C}(t_{1}))
\right\}.
\end{eqnarray*}
Hence, we may suppose that all the connections $\nabla \in \mathcal C_{\mathcal E}$ have the same eigendirection $q_{2}$ outside $L^{-1}(w_{\infty})$.

Reciprocally, given a connection $\nabla$ on $L\oplus L^{-1}(w_{\infty})$ satisfying $(\ref{-})$, we will show that $\nabla\in \mathcal C_{\mathcal E}$. In fact, suppose $\nabla$ is a connection on $(E, \{p_{1}, q_{2}\})$  satisfying (\ref{-}) and
\begin{displaymath}
\left\{ \begin{array}{ll}
p_{1} \nsubseteq  L_{t_{1}} \,\, \text{and} \,\, p_{1} \nsubseteq L^{-1}(w_{\infty})_{t_{1}}\\
q_{2} \nsubseteq L^{-1}(w_{\infty})_{t_{2}}.
\end{array} \right.
\end{displaymath}
We will prove that the second eigendirection of the residual part of $\nabla$ at $t_{2}$ which corresponds to $\nu_{2}^{+}$ lies in $L^{-1}(w_{\infty})$. In order to prove it, let us consider the apparent map with respect to $\mathcal L=L^{-1}(w_{\infty})$
$$
 \mathcal L \hookrightarrow   E \stackrel{\nabla}{\longrightarrow}  E\otimes\Omega_C^1(t_{1}+t_{2}) \longrightarrow (E/\mathcal L)\otimes \Omega_C^1(t_{1}+t_{2})=L(t_{1}+t_{2}).
$$
The zero divisor of the corresponding $\mathcal O_{C}$-linear map 
$$
\varphi_{\nabla}: \mathcal O_C\longrightarrow \mathcal L^{-1}\otimes L(t_{1}+t_{2})=\mathcal O_{C}(t_{2})
$$
defines an element of the linear system $\mathbb P{\rm {H}}^0(C,\mathcal O_{C}(t_{2}))=|t_{2}|$. This means that its zero divisor is exactly $t_{2}$, because our curve is elliptic. Consequently,  the residual part of $\nabla$ has  an eigendirection, say $p_{2}$, lying in  $L^{-1}(w_{\infty})$. Since $q_{2}$ corresponds to eigenvalue $\nu_{2}^{-}$, then $p_{2}$ corresponds to  $\nu_{2}^{+}$. It is enough to prove that $\nabla\in \mathcal C_{\mathcal E}$. 

Let us denote by $\mathcal E_{-}$ the parabolic bundle obtained by taking $q_{2}$ instead of $p_{2}$
$$
\mathcal E_{-} = (  L\oplus L^{-1}(w_{\infty}), \{p_{1}, q_{2}\} ).
$$
We have showed above that each $\nabla\in \mathcal C_{\mathcal E}$ can be seen as a connection on the parabolic bundle $\mathcal E_{-}$ where
\begin{eqnarray*}
\left\{ \begin{array}{ll}
{\rm Res}_{t_1}(\nabla)(p_{1}) = \nu_1^+ \cdot p_{1} \\
{\rm Res}_{t_2}(\nabla)(q_{2}) = \nu_2^- \cdot q_{2}
\end{array} \right.
\end{eqnarray*}
and vice versa.  Then $\mathcal C_{\mathcal E}$ can be identified with the affine vector space 
$$
\mathcal C_{\mathcal E} \simeq {\rm H}^0(C,{\rm End}(\mathcal E_{-})\otimes 	\Omega_C^1(D)).
$$
As we know $\mathcal E_{-}$ is simple, then the above cohomology group is a two dimensional vector space.  
\endproof

\subsubsection{Patching open charts}\label{sec:WholeModuli}

Let us suppose $\nu_{k}^{+}\neq \nu_{k}^{-}$ for each $k=1,2$. Given an element $(E, \nabla, \textbf{p}) \in {\rm Con}^{\nu}(C,D)$, the parabolic direction $\textbf{p}$ is nothing but the eigendirection for the residual part of $\nabla$ with respect to $\nu_{k}^{+}$. Under our hypothesis on $\nu$, the parabolic data is actually uniquely defined by the connection itself. Then we shall write just $(E,\nabla)$ instead of $(E, \nabla, \textbf{p})$. Each $\nabla$ has two pairs of ``positive'' and ``negative'' eigendirections 
\begin{eqnarray*}
\left\{ \begin{array}{ll}
{\textbf{p}}^{+}_{\nabla} = (p_{1}^{+}(\nabla), p_{2}^{+}(\nabla))\\
{\textbf{p}}^{-}_{\nabla} = (p_{1}^{-}(\nabla), p_{2}^{-}(\nabla))
\end{array} \right.
\end{eqnarray*}
defined by the eigenvalues $\nu_{k}^{+}$ and $\nu_{k}^{-}$, respectively. 

The moduli space ${\rm Con}^{\nu}(C,D)$ admits two forgetful maps with values in the moduli space of simple parabolic bundles
\begin{eqnarray*}
Bun^{+}: {\rm Con}^{\nu}(C,D) &\dashrightarrow& Bun(C,D)\\
		(E, \nabla) &\mapsto& (E, \textbf{p}^+_{\nabla})
\end{eqnarray*}
and 
\begin{eqnarray*}
Bun^{-}: {\rm Con}^{\nu}(C,D) &\dashrightarrow& Bun(C,D).\\
		(E, \nabla) &\mapsto& (E, \textbf{p}^-_{\nabla})
\end{eqnarray*}
They are complementary in the following sense. Indetermination points of $Bun^{+}$ are not indetermination points of $Bun^{-}$ and vice versa. Indeed, indetermination points of $Bun^{+}$ are of the form $(E, \nabla)$ such that $(E,\textbf{p}^{+}_{\nabla})$ is not simple. There are eight indecomposable not simple parabolic bundles. It follows from Proposition \ref{description} that the indetermination points of $Bun^{+}$ is a union of  eight two dimensional affine spaces. But we could define our moduli space of connections by using $\nu_{k}^{-}$ instead of $\nu_{k}^{+}$. This description for indetermination points of $Bun^{+}$ is completely symmetric for $Bun^{-}$. In fact, indetermination points of $Bun^{-}$ are of the form $(E, \nabla)$ such that $(E,\textbf{p}^{-}_{\nabla})$ is not simple.  Besides that, there is an open set of ${\rm Con}^{\nu}(C,D)$ where fibers of $Bun^{+}$ and $Bun^{-}$ are transverse to each other.

In order to describe our moduli space, one define the following open subsets
$$\begin{matrix}
\mathcal C&:=&\left\{ (E,\nabla) \in {\rm Con}^{\nu}(C,D)\,;\; (E, \textbf{p}^{+}_{\nabla})\in X_{<} \right\}\\
&\hfill =&\left\{ (E,\nabla) \in {\rm Con}^{\nu}(C,D)\,;\; (E, \textbf{p}^{-}_{\nabla})\in X_{<} \right\}\\
\mathcal C_{>}^+&:=&\left\{ (E,\nabla) \in {\rm Con}^{\nu}(C,D)\,;\; (E, \textbf{p}^{+}_{\nabla})\in X_{>} \right\}\\
\mathcal C_{>}^-&:=&\left\{ (E,\nabla) \in {\rm Con}^{\nu}(C,D)\,;\; (E, \textbf{p}^{-}_{\nabla})\in X_{>} \right\}
\end{matrix}$$
Note that $\mathcal C = {\rm Con}_<^{\nu}(C,D)$ and $\mathcal C_{>}^+ = {\rm Con}_>^{\nu}(C,D)$. 
The union $\mathcal C \cup \mathcal C_{>}^+$ is the biggest  open set where $Bun^{+}$ is well defined, i.e. its domain of definition is
$$
Dom(Bun^{+})= \mathcal C \cup \mathcal C_{>}^+
$$
We would like to remark that $(E, \textbf{p}^{+}_{\nabla})\in X_{<}$ if, and only if $(E, \textbf{p}^{-}_{\nabla})\in X_{<}$, because $X_{<}$ corresponds to parabolic bundles with underlying vector bundle $E=E_{1}$.  Then $\mathcal C \cup \mathcal C_{>}^-$ is the domain of definition of $Bun^{-}$
$$
Dom(Bun^{-})= \mathcal C \cup \mathcal C_{>}^-.
$$
Applying Corollary \ref{maior} with $\nu_{k}^{-}$ in place of $\nu_{k}^{+}$, one obtains the same statement for $\mathcal C_{>}^-$ instead of $\mathcal C_{>}^+$.  Hence, $Bun^{-}: \mathcal C \longrightarrow X_{>}$ is an affine $\mathbb C^{2}$-bundle over $X_{>}$. We summarize the above discussion and previous results in the following theorem. We still denote by $\Delta\subset \mathbb P^1\times \mathbb P^1$  the diagonal and $S:=(\mathbb P^1\times \mathbb P^1)\backslash\Delta$ its complement.

\begin{thm}\label{main2}
Assume $\nu_1^{\epsilon_1}+\nu_2^{\epsilon_2} \notin \mathbb Z$ for any $\epsilon_{k}\in \{+,-\}$. The moduli space  ${\rm Con}^{\nu}(C,D)$ is a union of three affine $\mathbb C^{2}$-bundles over $\mathbb P^{1}\times \mathbb P^{1}$
$$
{\rm Con}^{\nu}(C,D)= \mathcal C \cup \mathcal C_{>}^+ \cup \mathcal C_{>}^-.
$$
The intersection $\mathcal U := \mathcal C \cap \mathcal C_{>}^+ \cap \mathcal C_{>}^-$ is the set of pairs $(E, \nabla)$ where $E=E_{1}$ and there is no degree zero subbundle of $E_{1}$ passing through neither $\textbf{p}^{+}$ nor $\textbf{p}^{-}$. Moreover, there are fiber-preserving isomorphisms between affine $\mathbb C^{2}$-bundle over $\mathbb P^{1}\times \mathbb P^{1}$
$$
\Phi_{D}^{+}: \mathcal C_{>}^+\longrightarrow {\rm Con}_<^{\lambda}(C,D) \;\;\;\text{and}  \;\;\;\Phi_{D}^-: \mathcal C_{>}^-\longrightarrow {\rm Con}_<^{\gamma}(C,D)
$$
where $\lambda=(\lambda_{1}^{\pm}, \lambda_{2}^{\pm})$ and $\gamma=(\gamma_{1}^{\pm}, \gamma_{2}^{\pm})$  satisfy
\begin{displaymath}
\left\{ \begin{array}{ll}
\lambda_{k}^{+} = \nu_{k}^{+}- 1/2\\
\lambda_{k}^{-} = \nu_{k}^{-}+1/2
\end{array} \right.
\;\;
\text{and}
\;\;
\left\{ \begin{array}{ll}
\gamma_{k}^{+} = \nu_{k}^{+}+ 1/2\\
\gamma_{k}^{-} = \nu_{k}^{-}-1/2
\end{array} \right.
\end{displaymath}
for each $k\in\{1, 2\}$.
\end{thm}

\proof
The hypothesis $\nu_1^{\epsilon_1}+\nu_2^{\epsilon_2} \notin \mathbb Z$  implies that only indecomposable parabolic bundles appears in our moduli space (see Corollary \ref{cor criterion}). 
If $(E, \nabla)$ does not belong to $\mathcal C \cup \mathcal C_{>}^+$, then $(E, \nabla, \textbf{p}^{+}_{\nabla})$ is not simple. That is, $(E, \textbf{p}) =\mathcal E_{i,k}$ as in $(\ref{types})$  and exactly one of the parabolics $p_{1}^{+}(\nabla)$ or $p_{2}^{+}(\nabla)$ lies in the maximal subbundle $L^{-1}(w_{\infty})$ according with the square root $L_{i,k}$ of $\mathcal O_{C}(w_{\infty}-t_{i})$ (see Figure \ref{basepoint}). In particular, both parabolics $p_{1}^{-}(\nabla)$ and $p_{2}^{-}(\nabla)$ do not lie in $L^{-1}(w_{\infty})$ and then 
$$
(E,\{p_{1}^{-}(\nabla), p_{2}^{-}(\nabla)\})\in X_{>}.
$$
This implies that $(E, \nabla)\in \mathcal C_{>}^-$ and therefore one concludes that 
$$
{\rm Con}^{\nu}(C,D)= \mathcal C \cup \mathcal C_{>}^+ \cup \mathcal C_{>}^-.
$$
Now, let us describe the intersection $\mathcal U=\mathcal C \cap \mathcal C_{>}^+ \cap \mathcal C_{>}^-$. 
If $(E, \nabla)\in \mathcal C\cap \mathcal C_{>}^+$, then $E=E_{1}$ and there is no degree zero line bundle passing through $\textbf{p}^{+}_{\nabla}$, because
$$
\mathcal C\cap \mathcal C_{>}^+ = {\rm Con}_{<}^{\nu}(C,D) \cap {\rm Con}_{>}^{\nu}(C,D).
$$ 
The same conclusion can be done for $\textbf{p}^{-}_{\nabla}$ in place of $\textbf{p}^{+}_{\nabla}$ when  $(E, \nabla)\in\mathcal C_{>}^-$. 

The fiber-preserving isomorphism $\Phi_{D}^{+}$ is obtained by taking two positive elementary transformations with center at  $\textbf{p}^{+}_{\nabla}$ followed by twisting by $\mathcal O_{C}(-w_{\infty})$ (see Theorem \ref{second open} for details). The other fiber-preserving isomorphism $\Phi_{D}^{-}$ is defined similarly taking the elementary transformation at $\textbf{p}^{-}_{\nabla}$ instead of $\textbf{p}^{+}_{\nabla}$. 
\endproof

Recall that $S=(\mathbb P^{1}\times \mathbb P^{1}) \backslash \Delta$ and there  are projections 
\begin{eqnarray*}
\tau^{\epsilon}:S^{2} &\longrightarrow& \mathbb P^1\times \mathbb P^1\\
		((p_{1}^{+}, p_{1}^{-}), (p_{2}^{+}, p_{2}^{-})) &\mapsto& (p_{1}^{\epsilon}, p_{2}^{\epsilon})
\end{eqnarray*}
for $\epsilon\in \{+,-\}$, making $S^{2}$ a double $\mathbb C^{2}$-affine bundle over $\mathbb P^{1}\times \mathbb P^{1}$. Fibers of $\tau^{+}$ and $\tau^{-}$ are transverse between them.  

\begin{thm}\label{main3}
Assume $\nu_1^{\epsilon_1}+\nu_2^{\epsilon_2} \notin \mathbb Z$ for any $\epsilon_{k}\in \{+,-\}$. If $\nu_{k}^{+}-\nu_{k}^{-}\notin \{0, 1, -1\}$ for $k\in\{1, 2\}$, then the moduli space ${\rm Con}^{\nu}(C,D)$ of $\nu$-flat connections over the elliptic curve $C$, minus two points, is a union of three copies of $S^{2}$
$$
{\rm Con}^{\nu}(C,D) = S^{2}\cup_{\Psi^{+}} S^{2} \cup_{\Psi^{-}} S^{2}.
$$ 
Where $\Psi^{\pm}: S^{2} \dashrightarrow S^{2}$ are fiber-preserving isomorphisms outside a $(2,2)$ curve $\Gamma\subset \mathbb P^{1}\times \mathbb P^{1}$ isomorphic to $C$
\[
\xymatrix { 
S^{2}  \ar@{->}[dr]_{\tau^{-}} \ar@{-->}[rr]^{\Psi^{-}}  &    & S^2 \ar@{->}[dl]^{{\tau}^-} \ar@{->}[dr]_{\tau^{+}}  & & S^{2}.  \ar@{-->}[ll]_{\Psi^{+}} \ar@{->}[dl]^{{\tau}^+}\\
      &          \mathbb P^1\times \mathbb P^1 & &  \mathbb P^1\times \mathbb P^1
}
\]
\end{thm}
 
 \proof
The three copies of $S^{2}$ are given by Theorem \ref{main1}
$$
{\rm Con}^{\nu}(C,D)= \mathcal C \cup \mathcal C_{>}^+\cup \mathcal C_{>}^+.
$$
In fact, $\nu_{k}^{+}-\nu_{k}^{-}\neq 0$ implies that $\mathcal C\simeq S^{2}$ (see Theorem \ref{isomorphism affine}).  If $\nu_{k}^{+}-\nu_{k}^{-}\notin \{1, -1\}$, by the same reason, the two moduli spaces ${\rm Con}_<^{\gamma}(C,D)$ and ${\rm Con}_<^{\lambda}(C,D)$ of Theorem \ref{main1}, can be also identified with $S^{2}$. Then we get the following fiber-preserving isomorphisms 
$$
\Phi_{D}^{+}: \mathcal C_{>}^+\longrightarrow {\rm Con}_<^{\gamma}(C,D)\simeq S^{2} \;\;\;\text{and}  \;\;\;\Phi_{D}^{-}: \mathcal C_{>}^-\longrightarrow {\rm Con}_<^{\lambda}(C,D)\simeq S^{2}.
$$
We define $\Psi^{+}$ and $\Psi^{-}$ as inverse of $\Phi_{D}^{+}$ and $ \Phi_{D}^{-}$, respectively. To conclude the proof of theorem we remark that $\mathcal C_{>}^+={\rm Con}_>^{\nu}(C,D)$ intersect $\mathcal C={\rm Con}_<^{\nu}(C,D)$ outside the locus of pairs $(E_{1},\nabla)$ such that the respective parabolic bundle $(E_{1}, \textbf{p}^{+}_{\nabla})$ has two directions $\textbf{p}^{+}_{\nabla}=\{p_{1}^{+}, p_{2}^{+}\}$ lying in the same degree zero line bundle $L\in {\rm Jac}^0(C)\simeq C$.  This locus is a $(2,2)$ curve $\Gamma\subset \mathbb P^{1}\times \mathbb P^{1}$ parametrized by $C$.  The same argument works for $\mathcal C_{>}^-$ with $\textbf{p}^{-}_{\nabla}$ instead of $\textbf{p}^{+}_{\nabla}$ (see Figure \ref{3 copias}).
\endproof

\begin{center}
\begin{figure}[h]
\centering
\includegraphics[height=1.3in]{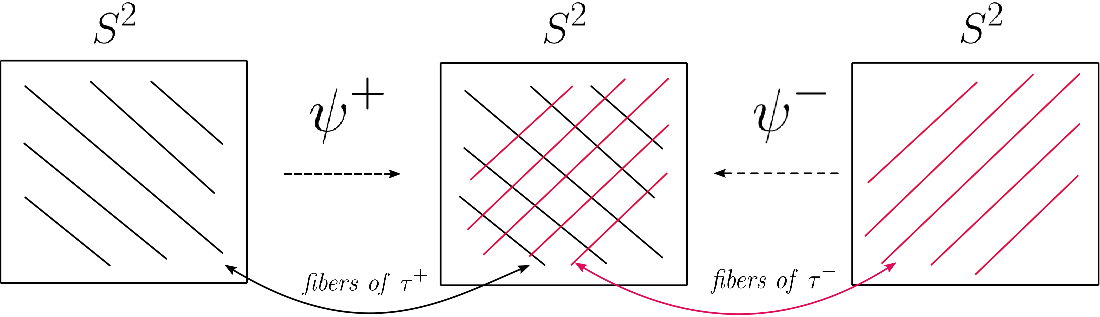}
\caption{Description of the moduli space}
\label{3 copias}
\end{figure}
\end{center}

\subsection{Apparent singularities}\label{sec:Apparent}

In this section, we will study the {\it apparent map} with respect to $\mathcal O_C \hookrightarrow E_1$ 
(see \cite{LS}). Each connection $\nabla$ on $E_1$ defines an $\mathcal O_C$-linear map:
$$
 \mathcal O_C \stackrel{\iota}{\hookrightarrow}   E_1 \stackrel{\nabla}{\longrightarrow}  E_1\otimes\Omega_C^1(D) \longrightarrow (E_1/\mathcal O_C)\otimes \Omega_C^1(D)
$$
where the last arrow  is defined by the quotient map from $E_1$ to $E_1/\iota(\mathcal O_C)$ (denoted $E_1/\mathcal O_C$) for short. That is, we shall consider the mapping
$$
\varphi_{\nabla}: \mathcal O_C\longrightarrow (E_1/\mathcal O_C)\otimes \Omega_C^1(D).
$$
Its zero divisor defines an element of the linear system 
$$
\mathbb P{\rm {H}}^0(C,(E_1/\mathcal O_C)\otimes \Omega_C^1(D)).
$$
Since $D=t_1+t_2$, $\det (E_1)=\mathcal O_C(w_{\infty})$ and $\Omega_C^1\simeq \mathcal O_C$ then 
$$
\mathbb P {\rm {H}}^0(C,(E_1/\mathcal O_C)\otimes \Omega_C^1(D))=|\mathcal O_C(w_{\infty}+t_1+t_2)|\simeq \mathbb P^2.
$$
Then one gets a map
$$
App: {\rm Con}_<^{\nu}(C,D) \longrightarrow |\mathcal O_C(w_{\infty}+t_1+t_2)|\simeq \mathbb P^2.
$$
On the other hand, we have a mapping
$$
Bun : {\rm Con}_<^{\nu}(C,D) \longrightarrow X_<\simeq \mathbb P^1_{z_1}\times \mathbb P^1_{z_2}
$$
which comes from the forgetful map.

\begin{definition}
Let $V\simeq\mathbb C^2\subset X_<\simeq \mathbb P^1\times \mathbb P^1$ be the open subset corresponding to parabolic bundles $(E_1,\textbf{p})$, $\textbf{p}=\{p_1,p_2\}$, such that $p_i\neq (1:t)$ for $i=1,2$, where $t\in\mathbb C$ is the first coordinate of $t_1=(t,\yt) \in C$. We have the decomposition:
$$\mathbb P^1_{z_1}\times \mathbb P^1_{z_2}=V\ \sqcup\ (\{z_1=t\} \cup \{z_2=t\})\ \ \ \text{and}\ \ \ 
\{z_1=t\} \cap \{z_2=t\}=:\{u_t\}.$$
\end{definition}

\begin{thm}\label{thm:apparentbirat}
If $\nu_1+\nu_2+1\neq 0$, then the map $Bun\times App$ defines a birational map
$$
Bun\times App : \overline{{\rm  {Con}}_<^{\nu}(C,D)} \longrightarrow (\mathbb P^1_{z_1}\times \mathbb P^1_{z_2})\times \mathbb P^2
$$
between $\mathbb P^2$-bundles over $\mathbb P^1_{z_1}\times \mathbb P^1_{z_2}$. Moreover, the following assertions hold true
\begin{enumerate}
\item its restriction to $V\subset \mathbb P^1_{z_1}\times \mathbb P^1_{z_2}$ is an isomorphism
$$
Bun\times App|_{Bun^{-1}(V)} : Bun^{-1}(V)\stackrel{\simeq}{\longrightarrow} V\times \mathbb P^2;
$$
\item if $u\in \{z_1=t\}\cup \{z_2=t\}\setminus \{u_t\}$, then the image of 
$$
Bun\times App|_{Bun^{-1}(u)} : \mathbb P^2{\longrightarrow} \mathbb P^2
$$
is a line;
\item if $\nu_1+\nu_2-1\neq 0$ and $u=u_t$, then $Bun\times App|_{Bun^{-1}(u_t)}$ is constant and its image is the point
$$
w_{\infty}+t_1+t_2\in |\mathcal O_C(w_{\infty}+t_1+t_2)|\simeq \mathbb P^2;
$$
\item if $\nu_1+\nu_2-1= 0$ and $u=u_t$, then $Bun^{-1}(u_t)$ lies in the indetermination locus of $Bun\times App$.  
\end{enumerate}
\end{thm}

\begin{remark}A phenomenon similar to case $(3)$ was observed in the genus zero case by Szil\'ard Szab\'o in \cite{Szabo}.
\end{remark}

\proof
Firstly, let us fix $\z \in U_0\subset  X_<$.  We will use  the explicit basis $\{\nabla^0, \Theta_1^0, \Theta_2^0\}$ given in Table \ref{basis} to compute the mapping 
\begin{eqnarray}\label{appzw}
App_{(\z)}: Bun^{-1}(\z)\simeq\mathbb P^2 \longrightarrow |\mathcal O_C(w_{\infty}+t_1+t_2)|\simeq \mathbb P^2
\end{eqnarray}
with respect to the subbundle  $\mathcal O_C(-2w_{\infty}) \hookrightarrow \mathcal O_C\oplus \mathcal O_C$ (see Remark \ref{trivial line bundle}).

In order to give explicit coordinates for $App$, we fix a basis of rational functions
$$
\left\{\frac{1}{(x-t)}, \frac{x}{(x-t)}, \frac{y}{(x-t)}\right\}
$$
for the vector space
$
{\rm {H}}^0(C, \mathcal O_C(w_{\infty}+t_1+t_2))
$
and denote by $\ba=(a_0: a_1: a_2)$ projective coordinates  with respect to that basis. 

Let $\bc=(c_0:c_1:c_2)\in \mathbb P^2$ be coordinates of $Bun^{-1}(\z)$ with respect to the basis $\{\nabla^0(\z), \Theta_1^0(\z), \Theta_2^0(\z)\}$ (see Section \ref{sec:universal}), that is,  each $\nabla(\z)\in Bun^{-1}(\z)$ writes
$$
\nabla(\z) = c_0\cdot\nabla^{0}(\z) + c_{1}\cdot \Theta_{1}^{0}(\z)+ c_{2}\cdot\Theta_{2}^{0}(\z).
$$
Since the subbundle $\mathcal O_C(-2w_{\infty}) \hookrightarrow \mathcal O_C\oplus \mathcal O_C$ corresponds to the explicit section 
$$s: C  \longrightarrow \mathcal O_C\oplus \mathcal O_C\ ;\ (x,y)\mapsto \begin{pmatrix}1\\ x\end{pmatrix}$$ 
as in Remark \ref{trivial line bundle}, we can compute the coordinates of the mapping $App_{\z}:\mathbb P^2_{\bc} \longrightarrow \mathbb P^2_{\ba}$: 
 $$
 App_{\z}(\bc)= a_0\cdot \frac{1}{(x-t)} + a_1\cdot  \frac{x}{(x-t)} + a_2 \cdot \frac{y}{(x-t)}
 $$
where
$${\footnotesize{
\begin{matrix}
a_0(\bc) &=& \yt(\nu_1(2z_1-t)+\nu_2(2z_2-t)-t)\cdot c_0 & - 4\yt z_1(z_1-t)\cdot c_1 & -4\yt z_2(z_2-t)\cdot c_2 \\
a_1(\bc) &=& -\yt(\nu_1+\nu_2-1)\cdot c_0 & + 4\yt(z_1-t)\cdot c_1 & + 4\yt(z_2-t)\cdot c_2 \\
a_2(\bc) &=& 2(\nu_1(z_1-t)-\nu_2(z_2-t))\cdot c_0 & - 4(z_1-t)^2\cdot c_1 & + 4(z_2-t)^2 \cdot c_2 
\end{matrix}
}
}$$
In other words, $App_{\z}:\mathbb P^2_{\bc} \longrightarrow \mathbb P^2_{\ba}$ is a  projective transformation defined by a matrix with determinant 
$$
\det(App_{\z})= -32 r^2(t-z_1)^2(t-z_2)^2(\nu_1+\nu_2+1).
$$
We promptly deduce that the rank drop exactly when $z_1=t$ or $z_2=t$. In addition, if $z_1=t$ and $z_2\neq t$, then $App_{\z}$ does not depend of $c_1$ and its matrix has rank two. The same happens if $z_1\neq t$ and $z_2=t$. If $z_1=z_2=t$, then the matrix has rank one. If $\nu_1+\nu_2-1\neq 0$, a straightforward computation shows that the image point is $(t:-1:0)$, 
which corresponds to a non-vanishing constant function in ${\rm {H}}^0(C, \mathcal O_C(w_{\infty}+t_1+t_2))$. This shows that the image point is $w_{\infty}+t_1+t_2\in |\mathcal O_C(w_{\infty}+t_1+t_2)|$ via the identification  $|\mathcal O_C(w_{\infty}+t_1+t_2)|\simeq \mathbb P{\rm {H}}^0(C, \mathcal O_C(w_{\infty}+t_1+t_2))$. Assertion $(4)$ can easily be verified.

We now consider the case where $$\Z=(Z_1,Z_2)=(1/z_1,1/z_2)\in U_{\infty}\subset  X_<.$$ Using the basis $\{\nabla^{\infty}, \Theta_1^{\infty}, \Theta_2^{\infty}\}$ instead of $\{\nabla^0, \Theta_1^0, \Theta_2^0\}$ one can find the coordinates of $App_{\Z}:\mathbb P^2_{\bC} \longrightarrow \mathbb P^2_{\ba}$ in terms of this basis: 
$$
 App_{\Z}(C_0\cdot\nabla^{\infty}+C_1\cdot\Theta_1^{\infty}+ C_2\cdot\Theta_2^{\infty})= a_0\cdot \frac{1}{(x-t)} + a_1\cdot  \frac{x}{(x-t)} + a_2 \cdot \frac{y}{(x-t)}
 $$
where
$${\footnotesize{
\begin{matrix}
a_0(\bC) =\hfill t\yt(1-\nu_1-\nu_2)\cdot C_0\hskip0.4cm  + 4\yt(1-tZ_1)\cdot C_1\hskip0.4cm  +4\yt(1-tZ_2)\cdot C_2\\
a_1(\bC) = \yt(\nu_1(2tZ_1-1)+\nu_2(2tZ_2-1)-1)\cdot C_0   + 4\yt Z_1(tZ_1-1)\cdot C_1   + 4\yt Z_2(tZ_2-1)\cdot C_2 \\
a_2(\bC) =\hfill  2t(\nu_1(tZ_1-1)-\nu_2(tZ_2-1))\cdot C_0\hskip0.4cm   + 4(tZ_1-1)^2\cdot C_1\hskip0.4cm    -4(tZ_2-1)^2 \cdot C_2 
\end{matrix}
}
}$$
It follows that 
$$
\det(App_{\Z})= -32(tZ_1-1)^2(tZ_2-1)^2(\nu_1+\nu_2+1).
$$
Since $U_0\cup U_{\infty}=(\mathbb P^1\times \mathbb P^1)\backslash \{(\infty,0), (0,\infty)\}$, to conclude the proof of theorem, it remains to consider the case where $u=(\infty,0)$ and $u=(0,\infty)$; we do not detail.
\endproof

\begin{thm}
If $\nu_1+\nu_2+1 = 0$, then the rational map 
$$
Bun\times App : \overline{{\rm  {Con}}_1^{\nu}(C,D)} \longrightarrow \mathbb P^1_{z_1}\times\mathbb P^1_{z_2}\times \mathbb P^2
$$
is non--dominant. Moreover,  $(Bun\times App)^{-1}(y)$ is one dimensional for general $y$  lying on the image of $Bun\times App$.
\end{thm}

\proof
Assume $\nu_1+\nu_2+1 = 0$. If $z_1\neq t$ and $z_2\neq t$ then 
$$
App_{\z}(\nabla^{0}(\z))  = -\frac{\nu_1}{2(z_1-t)}App_{\z}(\Theta_1^0(\z)) + \frac{\nu_1+1}{2(z_2-t)}App_{\z}(\Theta_2^0(\z)). 
$$
In addition, $App_{\z}(\Theta_1^0(\z))$ and $App_{\z}(\Theta_2^0(\z))$ do not coincide in $\mathbb P^2$. This shows that the image of $App_{\z}: \mathbb P^2 \longrightarrow \mathbb P^2$ is the line spanned by $App_{\z}(\Theta_1^0(\z))$ and $App_{\z}(\Theta_2^0(\z))$. 
\endproof

\subsection{Symplectic structure and Torelli phenomenon}\label{Sec:Symplectic}

Recall that the moduli space of connections ${\rm Con}^{\nu}(C,D)$ admits a canonical holomorphic symplectic structure,
i.e. a holomorphic non degenerate $2$-form $\omega$:  
$\omega\wedge\omega\not=0$ at any point of ${\rm Con}^{\nu}(C,D)$ (see for instance \cite{Simpson}).
Given any Lagrangian rational (or local) section $\nabla_0:{\rm Bun}^{\nu}(C,D)\to{\rm Con}^{\nu}(C,D)$,
then the reduction map
$$
  \xymatrix{ \nabla \ar@{|->}[rr] && \nabla-\nabla_0  \\ 
{\rm Con}^{\nu}(C,D) \ar@{-->}[rr] \ar[rd] && \Hig^{\nu}(C,D) \ar[ld] \\ &  {\rm Bun}^{\nu}(C,D)}
$$
is symplectic. Moreover, the following three maps are Lagrangian
$$\xymatrix{ &{\rm Con}^{\nu}_{<}(C,D) \ar[rr]^{App}  \ar[rd]_{Bun^-} \ar[dl]^{Bun^+} &&{\vert\mathcal O_C(w_{\infty}+t_1+t_2)\vert}\\
X_{<}   && X_{<} &}
$$
i.e. with Lagrangian fibers (see \cite{Simpson}). In particular, the section $\nabla^0$
defined in Section \ref{sec:universal} is Lagrangian since it is a fiber of $Bun^-$ (see Remark \ref{rk:characNabla0}).
We promptly deduce from the reduction map $\nabla_{\bc}\mapsto\nabla_{\bc}-\nabla^0$ applied to 
the universal connection $\nabla_{\bc}$ defined by (\ref{FormuleNablac})
that 
the symplectic structure on ${\rm Con}^{\nu}(C,D)$
is given by $\omega=dc_1\wedge dz_1+dc_2\wedge dz_2$
as for Higgs bundles in Theorem \ref{main1}.

The image of $\omega$ by the map $\Par:{\rm Con}^{\nu}_{<}(C,D)\to S$ defined in (\ref{FormuleParc})
is given by setting $c_i=\nu_i/2(z_i-\zeta_i)$, $i=1,2$, i.e. by
$$
\omega=-\frac{1}{2}\left\{\nu_1\frac{dz_1\wedge d\zeta_1}{(z_1-\zeta_1)^2}+\nu_2\frac{dz_2\wedge d\zeta_2}{(z_2-\zeta_2)^2}\right\}.
$$
In particular, we see that, taking into account the symplectic structure of the bundle $Bun:{\rm Con}^{\nu}_{<}(C,D)\to {\rm Bun}^{\nu}_{<}(C,D)$, we can recover the eigenvalues $\nu_1,\nu_2$ in the spirit of Torelli Theorem.

\begin{prop}\label{prop:Torelli}If we have an equivariant bundle symplectic isomorphism
$$
\xymatrix{
\left({\rm Con}^{\nu}_{<}(C,D),\omega\right)\ar[r]^\Phi_\sim \ar[d]^{Bun} & \left({\rm Con}^{\tilde{\nu}}_{<}(C,D),\tilde{\omega}\right) \ar[d]^{Bun} \\
{\rm Bun}_{<}(C,D)\ar[r]^\phi_\sim & {\rm Bun}_{<}(C,D)}
$$
Then $(\tilde{\nu}_1,\tilde{\nu}_2)=(\nu_1,\nu_2)$ or $(\nu_2,\nu_1)$.
\end{prop}

\begin{proof}Any automorphism $\phi$ of ${\rm Bun}_{<}(C,D)=X_{<}=\mathbb P^1_{z_1}\times\mathbb P^1_{z_2}$ writes
$$(z_1,z_2)\mapsto(\varphi(z_1),\psi(z_2))\ \ \ \text{or}\ \ \ (\varphi(z_2),\psi(z_1))$$
for some Moebius transformations $\varphi,\psi$. The map $\phi_{D}: (z_1,z_2)\mapsto(z_2,z_1)$ admits a lifting  
$$\Phi_D:((z_1,\zeta_1),(z_2,\zeta_2))\mapsto((z_2,\zeta_2),(z_1,\zeta_1))$$
satisfying  
$$\Phi_D^*\omega=-\frac{1}{2}\left\{\nu_2\frac{dz_1\wedge d\zeta_1}{(z_1-\zeta_1)^2}+\nu_1\frac{dz_2\wedge d\zeta_2}{(z_2-\zeta_2)^2}\right\}$$
(it permutes $\nu_1$ and $\nu_2$). Now, considering $\phi$ isotopic to the identity, we will prove that its lifting $\Phi$ preserves 
coefficients $\nu_1$ and $\nu_2$. First, observe that $\phi$ admits a lifting preserving $\nu_1$ and $\nu_2$, namely
$$\Phi:((z_1,\zeta_1),(z_2,\zeta_2))\mapsto((\varphi(z_1),\varphi(\zeta_1)),(\psi(z_2),\psi(\zeta_2))).$$
Indeed, we can decompose $\varphi$ and $\psi$ as a combination of $z\mapsto \alpha z$, $z\mapsto z+1$ and $z\mapsto1/ z$,
and easily check that, for each of these transformations, $\frac{dz\wedge d\zeta}{(z-\zeta)^2}$ is invariant,
and therefore $\omega$ as well. We end up the proof by showing that there are no other equivariant bundle
isomorphisms, i.e. if $\phi$ is the identity, then so is $\Phi$. Indeed, when we fix $z_2$, the $\Phi$
has to preserve the polar locus of $\omega$ (being symplectic), namely $\zeta_2=z_2$ and $\zeta_1=z_1$.
In restriction to $\zeta_2=z_2$ (recall we have fixed $z_2$) $\Phi$ induces an automorphism of 
$\mathbb P^1_{z_1}\times\mathbb P^1_{\zeta_1}$ preserving the diagonal $\zeta_1=z_1$ and fixing $z_1$.
It is therefore the identity. We easily conclude that $\Phi$ is the identity everywhere.
\end{proof}

It is interesting to make the link with the approach of \cite{LS}. There, a birational model of the moduli 
of parabolic bundles was introduced by using apparent map in restriction to Higgs fields (see 
\cite[Corollary 4.5]{LS}). Precisely, the restriction of apparent map to Higgs fields is given by
$$
 App_{\z}(c_{1}\cdot \Theta_{1}^{0}(\z)+ c_{2}\cdot\Theta_{2}^{0}(\z))=a_0\cdot \frac{1}{(x-t)} + a_1\cdot  \frac{x}{(x-t)} + a_2 \cdot \frac{y}{(x-t)}
 $$
where
$${\footnotesize{
\begin{matrix}
a_0 &=& -\yt( c_1z_1(t-z_1)+c_2z_2(t-z_2) ) \\
a_1 &=& \yt(c_1(t-z_1)+c_2(t-z_2)) \\
a_2 &=& c_1 (t-z_1)^2+c_2(t-z_2)^2
\end{matrix}
}
}$$
Therefore, the image $ App_{\z}\left(\Hig_<^{\nu}(C,D)\right)\subset {\vert\mathcal O_C(w_{\infty}+t_1+t_2)\vert}$
is given by $a_0b_0+a_1b_1+a_2b_2=0$ where
$${\footnotesize{
\begin{matrix}
b_0 &=& 2t-z_1-z_2 \\
b_1 &=& t(z_1+z_2)-2z_1z_2 \\
b_2 &=& r(z_1-z_2)
\end{matrix}
}
}$$
and we get a natural birational map
$$
Bun':{\rm Bun}_{<}(C,D)\to {\vert\mathcal O_C(w_{\infty}+t_1+t_2)\vert}^*\simeq\mathbb P^2_{\bb}\ ;
(z_1,z_2)\mapsto(b_0:b_1:b_2).
$$
This map blow-up the point $(z_1,z_2)=(t,t)$ and contracts the strict transform of the two lines 
$z_1=t$ and $z_2=t$. Another birational model of ${\rm Con}^{\nu}_{<}(C,D)$, more in the spirit of \cite{LS},
is given by 
\begin{equation}\label{DefAppBun}
App\times Bun':{\rm Con}^{\nu}_{<}(C,D)\stackrel{\sim}{\dashrightarrow}\mathbb P^2_{\ba}\times \mathbb P^2_{\bb}
\end{equation}
and we get (compare with \cite[Theorem 1.1]{LS})

\begin{prop}If $\nu_1+\nu_2+1\not=0$, then the map (\ref{DefAppBun}) is birational and the image of the symplectic 
form is given by 
$$\omega=d\eta\ \ \ \text{where}\ \ \ \eta=\left(\frac{\nu_1+\nu_2+1}{4}\right)\frac{a_0db_0+a_1db_1+a_2db_2}{a_0b_0+a_1b_1+a_2b_2}.$$
\end{prop}

The polar locus of $\omega$ is supported by the incidence variety $a_0b_0+a_1b_1+a_2b_2=0$.

\section{Appendix}

\subsection{Parabolic vector bundles}\label{def parabolic}

Let $C$ be a smooth irreducible projective complex curve and $D=t_{1}+\cdots +t_{n}$ be a reduced divisor supported on $n$ distinct points $\{t_{1},...,t_{n}\}\subset C$.  A rank two \textit{quasi-parabolic vector bundle} on $(C,D)$ is the data $(E, \textbf{p})$ where $E$ is a rank two holomorphic vector bundle over $C$ and $\textbf{p} = \{p_{1},..., p_{n}\}$ are given $1$-dimensional subspaces $p_{k}\subset E_{t_{k}}$ for $k\in \{1,...,n\}$. An isomorphism between quasi-parabolic vector bundles is, by definition, an isomorphism between underlying vector bundles preserving parabolic directions. A \textit{parabolic vector bundle} is a quasi-parabolic vector bundle together with a collection of weights $\mu = (\mu_{1},..., \mu_{n}) \in [0,1]^{n}$. It allows us to introduce a notion of stability in order to introduce a good moduli space. Given a line bundle $L\subset E$, the $\mu$-stability index of $L$ is the real number 
$$
{\rm Stab}(L):= \deg E - 2 \deg L + \sum_{p_{k}\neq L_{t_{k}}} \mu_{k} - \sum_{p_{k}= L_{t_{k}}} \mu_{k}.
$$ 

\begin{definition}\label{def stab}
A parabolic vector bundle $(E,\textbf{p})$ is called $\mu$-stable (resp. $\mu$-semistable) if for any rank one subbundle $L\subset E$, the following inequality holds
$$
{\rm Stab}(L) > 0 \;\;\; (resp. \; {\rm Stab}(L) \ge 0). 
$$
\end{definition}

It is well known, see \cite{MS}, that the moduli space of $\mu$-semistable parabolic vector bundles with fixed determinant line bundle is a normal irreducible projective variety. The open subset of $\mu$-stable parabolic bundles is smooth. We note that the stability index of $L\subset E$ is zero if, and only if, the weights lie along the following hyperplane in $[0,1]^{n}$
$$
\deg(E) - 2 \deg L + \sum_{p_{k}\neq L_{t_{k}}} \mu_{k} - \sum_{p_{k}= L_{t_{k}}} \mu_{k} = 0. 
$$
Each one of these hyperplanes is called a \textit{wall}. If we cut out $[0,1]^{n}$ by all possible walls, 
one gets in the complement of finitely many irreducible connected components, which are called \textit{chambers}. 
In each chamber, any $\mu$-semistable parabolic vector bundle is $\mu$-stable, 
and the moduli space is constant, that is, it is independent of $\mu$. Nevertheless, it can be empty.  
When we have two adjacent chambers separated by a wall, then there is a locus of $\mu$-stable parabolic bundles 
in each chamber that became unstable when we cross the wall. Along the wall, 
we identify strictly semistable parabolic bundles $(E,\textbf{p})$ and $(E',\textbf{p'})$ 
with ${\rm gr}(E,\textbf{p})= {\rm gr}(E',\textbf{p'})$, see \cite[Section 4]{MS}. Over the projective line,  
the description  of the moduli space of quasi-parabolic bundles has been done in \cite{Bauer}. 

\subsection{Elementary transformations} In the construction of the moduli space of quasi-parabolic bundles, the determinant line bundle is fixed. Actually, up to twists and elementary transformations, we can choose the determinant bundle
arbitrarily, for instance the trivial line bundle $\mathcal O_{C}$. Twists preserve the parity of the determinant and elementary transformations change it.  We start by recalling what is an elementary transformation as well as its main properties.  Given $t\in C$ and a direction $p\subset E_{t}$ the vector bundle $E^{-}$ is defined by the following exact sequence of sheaves 
$$
0\longrightarrow E^{-} {\longrightarrow}  E {\longrightarrow}  E/p \longrightarrow 0.
$$
where $p$ appearing above is considered as a sky-scrapper sheaf. The new parabolic direction $p^{-} \subset E^{-}$ is the kernel of the morphism $E^{-} \longrightarrow E$. By identifying sections of $E$ and $E^{-}$ outside of $t$, one obtains a birational bundle transformation
$$
elem^{-}_{t} :E \dashrightarrow E^{-}
$$
with center at $p$, which is an isomorphism outside $t$. We shall say that it is a \textit{negative elementary transformation}. At a neighborhood of $t$, it can be described as follows. We can choose a local trivialization  $E|_{U}\simeq U\times \mathbb C^{2}$ such that \begin{eqnarray*}
p=\left(
\begin{array}{ccc} 
0   \\
1 \\
\end{array}
\right)
\subset E_{t}.
\end{eqnarray*}
and $elem^{-}_{t}: E|_{U}  \dashrightarrow E^{-}|_{U}$  is given by 
\begin{eqnarray*}
elem^{-}_{t}(x, Y) =
\left(
\begin{array}{ccc} 
1/x & 0  \\
0 & 1  \\
\end{array}
\right)\cdot Y.
\end{eqnarray*}
From the point of view of ruled surfaces, it corresponds to a flip with center at $[p]\in \mathbb P E$, that is, a blow up at $[p]$ followed by a contraction of the old fiber. If the direction $p$ is contained in a line bundle $L\subset E$, then it is left unchanged and one obtains a line bundle $L\subset E^{-}$. If $L\subset E$ does not contain $p$, then we get $L^{-}\subset E$ where 
$$
L^{-}= L\otimes \mathcal O_{C}(-t). 
$$
In addition, we have the following property 
$$
\det (E^{-}) = \det (E) \otimes \mathcal O_{C}(-t).
$$

When we perform an elementary transformation, the stability condition is preserved after an appropriate 
modification of weights. If $(E, \textbf{p})$ is $\mu$-stable and we perform an elementary transformation 
$elem^{-}_{t_{k}}$, then $(E^{-}, \textbf{p}^{-})$ is $\mu'$-stable where
\begin{eqnarray*}
\left\{ \begin{array}{ll}
\mu_{k}' = 1- \mu_{k} \\
\mu_{j}' = \mu_{j} \;\; j\neq k. 
\end{array} \right.
\end{eqnarray*}
If ${\rm Bun}_{L}^{\mu}(C,D)$ denotes the moduli space of $\mu$-semistable parabolic vector bundles with fixed determinant line bundle $L$, then $elem^{-}_{t_{k}}$ defines an isomorphism between moduli spaces
\begin{eqnarray*}
elem^{-}_{t_{k}}: {\rm Bun}_{L}^{\mu}(C,D) &\longrightarrow& {\rm Bun}_{L(-t_{k})}^{\mu'}(C,D)\\
(E, \textbf{p}) &\mapsto& elem^{-}_{t_{k}}(E, \textbf{p}).
\end{eqnarray*}

We can  define a \textit{positive elementary transformation} $elem^{+}_{t}$ as 
 $$
 elem^{+}_{t} := \mathcal O_{C}(t)\otimes elem^{-}_{t} : E\dashrightarrow E^{+}
 $$
 where $E^{+} = E^{-}\otimes \mathcal O_{C}(t)$. It is the inverse of $elem^{-}_{t}$. 
 As before, stability condition is preserved by elementary positive transformations, 
 with the same modification of weights. 

\subsection{Endomorphisms of quasi-parabolic vector bundles}

The space of global endomorphisms of a rank two vector bundle 
as well as the automorphism group are well known, see for example \cite[Theorem 1]{Ma}. 
In this section, we study the space of traceless endomorphisms of a quasi-parabolic bundles over an elliptic curve. 

Let $(E,\textbf{p})$ be a quasi-parabolic bundle. We say that an endomorphism $\sigma \in {\rm End}(E)$ fixes 
the parabolic structure $\textbf{p}$  if $\sigma(p_k) \subset p_k$, for all $k=1,..., n$. Let ${\rm End}(E,\textbf{p})$ 
be the vector space of endomorphisms fixing the parabolic structure $\textbf{p}$ and ${\rm End}_0(E,\textbf{p})$ 
its subspace of traceless endomorphisms. We notice that we have a canonical decomposition:
$$
{\rm End}(E,\textbf{p})= < Id > \oplus \;{\rm End}_0(E,\textbf{p}) 
$$
where $Id \in {\rm End}(E,\textbf{p})$ is the identity. This follows from the following simple remark: if $A$ is a $2\times 2$ matrix then
$$
A = \frac{{\rm tr}(A)}{2}\cdot Id+  \left (A-\frac{{\rm tr}(A)}{2}\cdot Id\right ).
$$

\begin{lemma}\label{endom}
Let $E=\mathcal O_C\oplus L$ be a rank $2$ bundle over a projective smooth curve $C$ with $\mathcal L$ of nonnegative degree. The following statements hold true.
\begin{enumerate}
\item[(a)] If $L \ncong \mathcal O_C$ then
\begin{eqnarray*}
{\rm End}_0(E) = \left\{
\left(
\begin{array}{ccc} 
a & 0  \\
\gamma & -a  \\
\end{array}
\right)
\;;\;\; a\in\mathbb C,\; {\color{black}\gamma=\{\gamma_i\}\in {\rm H}^0(C,L)}
\right\}
\end{eqnarray*}
\item[(b)] If $\mathcal L = \mathcal O_C$ then
\begin{eqnarray*}
{\rm End}_0(E) = \left\{
\left(
\begin{array}{ccc} 
a & b  \\
c & -a  \\
\end{array}
\right)
\;;\;\; a,b,c \in\mathbb C
\right\}
\end{eqnarray*}
\end{enumerate}
\end{lemma}

\proof
We leave the proof to the reader; see for example \cite[Theorem 1]{Ma}.
\endproof

We recall that a quasi-parabolic bundle $(E,\textbf{p})$ is decomposable if there exists a decomposition $E=L_1\oplus L_2$ such that each parabolic direction is contained in $L_1$ or $L_2$. In this case, we write 
$$
(E,\textbf{p}) = (L_1,{\bf{p_1}}) \oplus  (L_2,{\bf{p_2}}).
$$

\subsection{The case of elliptic curves}
In what follows we will determine the traceless endomorphisms of an indecomposable quasi-parabolic bundle over an elliptic curve. This will be useful to assure existence of logarithmic connections. Before that, we shall give one example.   

\begin{example}
Let $(\mathcal O_C\oplus \mathcal O_C (t_1), \{p_1,p_2\})$ be a quasi-parabolic bundle over an elliptic curve $(C,\{t_1,t_2\})$ with 
\begin{enumerate}
\item $p_1$  outside $\mathcal O_C|_{t_1}$ and $\mathcal O_C (t_1)|_{t_1}$; and
\item $p_2 \subset \mathcal O_C (t_1)|_{t_2}$.
\end{enumerate}
It is indecomposable because any subbundle given by an embedding of the trivial bundle $\mathcal O_C\hookrightarrow \mathcal O_C\oplus \mathcal O_C (t_1)$ corresponds to a section of $\mathbb P (\mathcal O_C\oplus \mathcal O_C (t_1))$ which has $(1:0)$ as a base point over $t_1$. In fact, since $C$ is elliptic $h^0 (\mathcal O_C (t_1))=1$. On the other hand, {\color{black}if $\gamma\in {\rm H}^0(C,\mathcal O_C (t_1))$ is a section} which corresponds to the divisor $D=t_1$ then 
 \begin{eqnarray*}
{\rm End_0}(E,\textbf{p}) = \left\{
\left(
\begin{array}{ccc} 
0 & 0  \\
c\gamma & 0  \\
\end{array}
\right)
\;;\;\; c \in\mathbb C
\right\}.
\end{eqnarray*}
\end{example}

\begin{prop}\label{case decomposable}
Let $E=\mathcal O_C \oplus L$, $\deg (L)\ge 0$, be a decomposable rank $2$ bundle over an elliptic curve $C$. Assume $(E,\textbf{p})$ is indecomposable but not simple, i.e. ${\rm End}_0(E,\textbf{p}) \not= \{ 0 \}$. 
Then we are in the following case
\begin{enumerate}
\item ${\rm End}_0(E,\textbf{p}) = \mathbb C$, 
\item the support $D$ of parabolics splits as $D=D_0+D_1$ with $\deg(D_1)>0$,
\item $L\simeq\mathcal O_C(D_1)$ (and has $>0$ degree),
\item parabolics over $D_0$ are lying on $L$,
\item parabolics over $D_1$ are outside $L$ and generic.
\end{enumerate}
Here, generic means that there is no embedding $\mathcal O_C\hookrightarrow E$ passing through 
all parabolics over $D_1$, i.e. $(E,\textbf{p})$ is indecomposable.
\end{prop}

Observe that, by Lemma \ref{endom}, one can find an embedding $\mathcal O_C\hookrightarrow E$ passing through
all but one parabolics over $D_1$. In particular, given the decomposition $D=D_0+D_1$, there is a unique
such parabolic bundle $(E,\textbf{p})$ up to isomorphism.

\proof
If $L = \mathcal O_C$  and $(E,\textbf{p})$ is indecomposable, then there are at least $3$ parabolics which not two of them lies in the same embedding of $\mathcal O_C   \hookrightarrow \mathcal O_C\oplus \mathcal O_C$. Then it follows from Lemma \ref{endom} item $(b)$ that 
$$
{\rm End}_0(\mathcal O_C\oplus \mathcal O_C,\textbf{p}) = \{0\}.
$$
Let us suppose $L \ncong \mathcal O_C$ and let $\phi$ be a traceless endomorphism that fixes the parabolics. Lemma \ref{endom} item $(a)$ implies that we can choose a covering of $C$ and trivializations such that the vector $e_1$ generates $\mathcal O_C$, $e_2$ generates $L$ and
\begin{eqnarray*}
\phi =
\left(
\begin{array}{ccc} 
a & 0  \\
\gamma & -a  \\
\end{array}
\right)
\end{eqnarray*}
where $a\in \mathbb C$, $\gamma\in {\rm H}^0(C,L)$. 
If $a\neq 0$, then we see that  the locus of fixed points of $\phi$ outside $L$ defines a section of $E$ 
\begin{eqnarray*}
\left(
\begin{array}{ccc} 
1   \\
\frac{\gamma}{2a}   \\
\end{array}
\right)
\end{eqnarray*}
which generates a subbundle $\mathcal O_C   \hookrightarrow E$ containing all parabolics outside $L$, 
showing that $(E,\textbf{p})$ must be decomposable. We can thus assume $a=0$.

Since $\phi$ preserves the parabolic structure, we have for each parabolic $p_k$
\begin{itemize}
\item either $p_k$ is in the subbundle $L\subset E$,
\item or the support $t_k$ of $p_k$ is a zero of $\gamma$. 
\end{itemize}
Therefore, we can decompose $D=D_0+D_1$ where $D_0$ is the support of those $p_k$'s lying in $L$.

Since $(E,\textbf{p})$ is indecomposable, there is no embedding $\mathcal O_C\hookrightarrow E$ passing through 
all parabolics over $D_1$. Assume that our decomposition $E=\mathcal O_C \oplus L$ maximize the number
of parabolics lying on the first factor. Set $D_1=D_1'+D_1''$ with $D_1'$ supporting those parabolics in $\mathcal O_C$
and $\deg(D_1'')>0$.
By maximality, we have that each section $\varphi\in H^0(C,L)$ which vanishes on $D_1'$ automatically
vanishes on $D_1''$. In other words, each section of $L\otimes\mathcal O_C(-D_1')$ automatically
vanishes on $D_1''$. On the other hand, we know that $L\otimes\mathcal O_C(-D_1')$ admits a non zero
section defined by $\gamma$. But on the elliptic curve $C$, the only linear systems with base points
are of the form $\vert t_k\vert$, meaning that $L\otimes\mathcal O_C(-D_1')\simeq\mathcal O_C(t_k)$ and 
$D_1''=t_k$ reduces to a single point. We therefore deduce that $L\simeq\mathcal O_C(D_1)$ and 
all parabolics but $p_k$, over $D_1$, are lying on the first factor $\mathcal O_C\subset E$.\endproof

Let $E_0$ be the unique indecomposable rank $2$ bundle, over an elliptic curve, with trivial determinant and having $\mathcal O_C$ as maximal subbundle. It corresponds to the non trivial extension defined by the following exact sequence 
$$
0\longrightarrow \mathcal O_C {\longrightarrow}  E_0{\longrightarrow}  \mathcal O_C \longrightarrow 0.
$$

\begin{prop}\label{case indecomposable}
Let $(E_0,\textbf{p})$ be a quasi-parabolic bundle, where $E_0$ is the indecomposable bundle as above.
\begin{enumerate}
\item If all parabolics lie in the maximal subbundle $\mathcal O_C\hookrightarrow E_0$ then
$$
{\rm End}_0(E_0,\textbf{p}) \simeq \mathbb C
$$

\item  If there exists at least one parabolic outside $\mathcal O_C\hookrightarrow E_0$ then
$$
{\rm End}_0(E_0,\textbf{p}) = \{0\}.
$$
\end{enumerate}
\end{prop}

\proof
The traceless endomorphism space of $E_0$ is given by
\begin{eqnarray*}
{\rm End_0}(E_0) = \left\{
\left(
\begin{array}{ccc} 
0 & b  \\
0 & 0  \\
\end{array}
\right)
\;;\;\; b\in\mathbb C
\right\}.
\end{eqnarray*}
Here we are considering that in local charts $U_i\subset C$, $e_1$ generates the maximal subbundle $\mathcal O_C$.  Then any traceless endomorphism leaves $\mathcal O_C$ invariant. Also if a parabolic direction outside $\mathcal O_C$ is fixed, one gets $b=0$. This is enough to conclude the proof of  proposition.  
 \endproof

\begin{prop}\label{totheorem}
Let $(E,\textbf{p})$ be indecomposable but not simple rank $2$ quasi-parabolic bundle over an elliptic curve $C$. Assume $E$ has trivial determinant line bundle. Then, up to elementary transformations and twists, we can assume $E=E_{0}$ with all the parabolic lying in the maximal subbundle $\mathcal O_C\hookrightarrow E_0$. 
\end{prop}

\proof
The proof follows essentially from Proposition \ref{case decomposable} and Proposition \ref{case indecomposable}. Suppose $E$ is decomposable,  $E=M^{-1}\oplus M$, $\deg M = k \ge 0$. Then $\mathbb PE=\mathbb P(\mathcal O_{C}\oplus M^{2})$. From Proposition \ref{case decomposable}, the support $D$ of parabolics splits as $D=D_0+D_1$ with $\deg(D_1)>0$ and $\mathcal O_{C}(D_{1})=M^{2}$. Parabolics over $D_1$ are outside $M\hookrightarrow E$ (which corresponds to $M^{2}\hookrightarrow\mathcal O_{C}\oplus M^{2})$) and generic. Parabolics over $D_0$ are lying on $M$. After a composition $elem_{D_{1}}$ of $2k$ elementary transformation over $D_{1}$ and twist (to get trivial determinant), we arrive in $E_{0}$ with all the parabolics lying in the maximal subbundle. If $E$ is indecomposable, the conclusion follows from Proposition \ref{case indecomposable}. 
\endproof

\subsection{Moduli space of connections}\label{moduli space}
Let $C$ be a smooth projective curve and $D=t_1+\cdots+t_n$ be a reduced divisor on $C$, $n\ge 1$.  We will fix the data in order to introduce the moduli space of connections. Firstly, let us fix a degree $d$ line bundle $L_{0}$ over $C$.  We also set a local exponent $\nu\in\mathbb C^{2n}$ satisfying the Fuchs relation
\begin{eqnarray*}
 \sum_{k=1}^n (\nu_k^+ + \nu_k^-)+d=0
\end{eqnarray*}
and the generic condition 
$
\nu_1^{\epsilon_1}+\cdots+\nu_n^{\epsilon_n} \notin \mathbb Z 
$
for any  $\epsilon_k \in \{+,- \}$, to avoid reducible connections. Let $\zeta: L_{0} \longrightarrow L_{0}\otimes \Omega^1_C(D)$ be any fixed  rank one logarithmic connection on $L_{0}$ satisfying 
$$
{\rm Res}_{t_k}(\zeta)=\nu_k^+ + \nu_k^-
$$
for all $k=1,...,n$.
We denote by ${\rm Con}^{\nu}(C,D)$ the moduli space of triples $(E,\nabla, \textbf{p})$ where 
\begin{enumerate}
\item $(E, \textbf{p})$ is a rank $2$ quasi-parabolic vector bundle over $(C,D)$ having $L_{0}$ as determinant bundle;
\item $\nabla: E \longrightarrow E\otimes \Omega_C^1(D)$ is a logarithmic connection on $E$ with polar divisor $D$, having $\nu$ as local exponents and ${\rm tr}(\nabla) = \zeta$;
\item two triples $(E,\nabla, \textbf{p})$ and $(E',\nabla',\textbf{p})$ are equivalent when there is an isomorphism between quasi-parabolic bundles $(E,\textbf{p})$ and $(E',\textbf{p}')$ conjugating $\nabla$ and $\nabla'$. 
\end{enumerate}

Actually, in order to obtain a good moduli space we need a stability condition. A tripe $(E,\nabla, \textbf{p})$ is called $\mu$-stable (resp. $\mu$-semistable) if for any $\nabla$-invariant line bundle $L \subset E$, we have 
$$
{\rm Stab}(L) > 0 \;\;\; (resp. \; {\rm Stab}(L) \ge 0)
$$
(see Definition \ref{def stab}). But an invariant line bundle $L$ would force a relation
$$
\nu_1^{\epsilon_1}+\cdots+\nu_n^{\epsilon_n} + \deg(L) = 0 
$$
which is obtained by applying Fuchs relation to the restriction $\nabla|_{L}$. This contradicts our hypothesis on $\nu$. Therefore under generic condition on the local exponent all the connections arising in our moduli space are stable. It follows from \cite[Theorem 3.5]{Ni} that ${\rm Con}^{\nu}(C,D)$ is a quasi-projective variety. 

A priori, ${\rm Con}^{\nu}(C,D)$ depends on the choice of $L_{0}$. But up to twists, we can go into either the even case $L_{0}= \mathcal O_{C}$ or the odd case $L_{0}= \mathcal O_{C}(t)$. In fact, given a rank one logarithmic connection $\eta: L \longrightarrow L\otimes \Omega^1_C(D)$ with local exponents $(k_{1},...,k_{n})$, we can define a twisting map 
\begin{eqnarray*}
\otimes (L, \eta): {\rm Con}^{\nu}(C,D) &\longrightarrow& {\rm Con}^{\nu'}(C,D)\\
               (E,\nabla, \textbf{p}) &\mapsto& (E\otimes L,\nabla\otimes \eta, \textbf{p})
\end{eqnarray*}
where $\nu' = (\nu_{1}^{\pm}+k_{1}, ... , \nu_{n}^{\pm}+k_{n})$. Such map is an isomorphism between moduli spaces, in particular our moduli space only depend on differences $\nu_{k}^{+} - \nu_{k}^{-}$. In the even case, we can assume that $(L_{0}, \zeta)=(\mathcal O_{C}, d)$ where $d$ means the trivial rank one connection. In the odd case, one may suppose $(L_{0}, \zeta)=(\mathcal O_{C}(t), d - \frac{dx}{x-t})$. 

Now let us deal with elementary transformations. When we perform a transformation $elem^{-}_{t_{k}} :(E, \textbf{p}) \dashrightarrow (E^{-}, \textbf{p}^{-})$ the new connection $\nabla^{-}$ on $E^{-}$ has local exponents
$$
(\nu_{k}^{+}, \nu_{k}^{-})' = (\nu_{k}^{-}+1, \nu_{k}^{+})
$$ 
and the other $\nu_{j}$, $j\neq k$, are left unchanged. Finally, we can go from the odd to the even case by performing one negative elementary transformation $elem^{-}_{t_{n}}$.  

In the case we are interested in, $C$ is supposed to be an elliptic curve and $D=t_{1}+t_{2}$. For computation, we can assume $C\subset \mathbb P^2$  is the smooth projective cubic curve
\begin{eqnarray}\label{ellipticbis}
zy^2 = x(x-z)(x-\lambda z)
\end{eqnarray}
with $\lambda \in \mathbb C$, $\lambda \neq 0,1$. And by the above digression one can set  $L=\mathcal O_C(w_{\infty})$, $w_{\infty} = (0:1:0)$. As local exponents, we can take 
\begin{eqnarray}\label{nu fixed2}
(\nu_1^{\pm},\nu_2^{\pm})=\left (\pm \frac{\nu_{1}}{2} - \frac{1}{2}, \pm \frac{\nu_{2}}{2}\right ).
\end{eqnarray}
We note that for each $k\in\{1,2\}$, the condition $\nu_k^{+}=\nu_k^{-}$ is equivalent to $\nu_k=0$.

\bigskip

\noindent\textsc{{\bf Acknowledgements}}.
The first author would like thanks the Institut de Recherche en Math\'ematique de Rennes, IRMAR, for the hospitality and support.

\end{document}